\documentclass[article]{siamart190516}
\usepackage{url}
\usepackage{amssymb}
\usepackage{amsmath}
\usepackage{stmaryrd}
\usepackage{siunitx}
\usepackage{enumitem}
\usepackage{commath}
\usepackage{graphicx,subfig}
\usepackage[numbers,square]{natbib}
\usepackage{tikz}
\makeatletter
\newcommand{\tnorm}{\@ifstar\@tnorms\@tnorm}
\newcommand{\@tnorms}[1]{%
  \left|\mkern-1.5mu\left|\mkern-1.5mu\left|
   #1
  \right|\mkern-1.5mu\right|\mkern-1.5mu\right|
}
\newcommand{\@tnorm}[2][]{%
  \mathopen{#1|\mkern-1.5mu#1|\mkern-1.5mu#1|}
  #2
  \mathclose{#1|\mkern-1.5mu#1|\mkern-1.5mu#1|}
}
\makeatother

\newcommand{\seminorm}[1]{%
  \left|
   #1
  \right|
}

\newcommand{\LRa}[1]{\left\langle #1 \right\rangle}
\newcommand{\lra}[1]{\langle #1 \rangle}
\newcommand{\nb}{\boldsymbol{n}}
\newcommand{\Phib}{\boldsymbol{\Phi}}
\newcommand{\qb}{\boldsymbol{q}}
\newcommand{\pb}{\boldsymbol{p}}
\newcommand{\rb}{\boldsymbol{r}}
\newcommand{\ubar}{\bar{u}}

\newcommand{\Vb}{\boldsymbol{V}}
\newcommand{\vbar}{\bar{v}}
\newcommand{\zbar}{\bar{z}}
\newcommand{\wb}{\boldsymbol{w}}
\newcommand{\Xb}{\boldsymbol{X}}

\newtheorem{remark}{Remark}[section]

\title{A hybridized discontinuous Galerkin method for Poisson-type
  problems with sign-changing coefficients \thanks{\funding{SR
      gratefully acknowledges support from the Natural Sciences and
      Engineering Research Council of Canada through the Discovery
      Grant program (RGPIN-05606-2015).}}}
\author{ Jeonghun~J.~Lee\thanks{Department of Mathematics, Baylor
    University, Waco, TX, USA (\url{jeonghun_lee@baylor.edu}),
    \url{http://orcid.org/0000-0001-5201-8526}} \and
  Sander~Rhebergen\thanks{Department of Applied Mathematics,
    University of Waterloo, Canada (\url{srheberg@uwaterloo.ca}),
    \url{http://orcid.org/0000-0001-6036-0356}}}
\headers{HDG for Poisson-type problems with sign-changing
  coefficients}{Jeonghun~J.~Lee and Sander Rhebergen}
\begin{document}
\maketitle
\begin{abstract}
  In this paper, we present a hybridized discontinuous Galerkin (HDG)
  method for Poisson-type problems with sign-changing coefficients. We
  introduce a sign-changing stabilization parameter that results in a
  stable HDG method independent of domain geometry and the ratio of
  the negative and positive coefficients. Since the Poisson-type
  problem with sign-changing coefficients is not elliptic, standard
  techniques with a duality argument to analyze the HDG method cannot
  be applied. Hence, we present a novel error analysis exploiting the
  stabilized saddle-point problem structure of the HDG
  method. Numerical experiments in two dimensions and for varying
  polynomial degree verify our theoretical results.
\end{abstract}
\section{Introduction}
\label{sec:introduction}

Partial differential equations (PDE) with sign-changing coefficients
are used to model physical phenomena of meta-materials, for example,
wave transmission problems between classical materials and
meta-materials. Therefore, there is an emerging interest in the
development of numerical methods for such PDEs.  However, the bilinear
forms for problems with sign-changing coefficients are not coercive,
so standard techniques relying on coercivity cannot be applied. The
lack of coercivity indeed poses fundamental challenges to study the
well-posedness of such PDEs, as well as the development of numerical
methods.

A popular approach in the study of numerical methods for PDEs with
sign-changing coefficients is the ``$T$-coercivity'' approach. This
approach utilizes a linear operator $T$ that recovers the coercivity
of the bilinear form. The existence of such an operator $T$ is a
natural assumption because the existence of $T$ is equivalent to the
well-posedness of the PDE in the sense of
Banach--Ne\v{c}as--Babu\v{s}ka in the functional analysis framework
\cite{Chesnel:2013}. The application of the $T$ operator to develop
numerical methods was first proposed in \cite{Zwolf:2010}. Since then,
the $T$-coercivity approach has been used in the development of
numerical methods for Poisson-type problems \cite{Chesnel:2013},
Helmholtz-type problems \cite{Chung:2013}, and Maxwell-type problems
\cite{Zwolf:2010, Dhia:2014, Chesnel:2014}.

In numerical methods that use $T$-coercivity, the $T$ operator cannot
be used directly, because the domain and the range of $T$ are not the
discrete spaces in general. For this reason, $T_h$ is introduced, a
discrete approximation of $T$, which has the coercivity recovery
property in the discrete setting. The existence of such $T_h$ is
non-trivial and difficult to guarantee because the form of $T$ is
unknown in general. For general geometry and meshes sufficient
conditions for the existence of $T_h$ are proposed in
\cite{Chesnel:2013}, but these include non-quantitative constants such
as a norm of the discrete interpolation of $T$ and the ratio of the
positive and negative coefficients (the contrast). It is furthermore
shown in \cite{Chesnel:2013} that locally symmetric meshes across the
transmission interface, where the coefficient changes sign, improve
the quality of numerical solutions. However, the generation of such
meshes is a non-trivial restriction in mesh generation for complex
interfaces.

Hybridization (or static condensation) was originally introduced for
mixed finite element methods to reduce computational cost
\cite{veubeke}, and it was further analyzed in
\cite{Arnold:1985}. More recently, hybridization was combined with the
discontinuous Galerkin (DG) method. This resulted in the hybridizable
discontinuous Galerkin (HDG) method which was shown to be more
efficient than standard DG methods. The HDG method was systematically
presented in \cite{Cockburn:2009} for elliptic partial differential
equations. Since its introduction, it has been extended to many
different problems. These include HDG methods for the Helmholtz
problem, e.g., \cite{Cui:2013, Griesmaier:2011} and the Maxwell
problem, e.g., \cite{Chen:2017, feng:2016}.

In this paper, we develop a numerical method for Poisson-type problems
with a sign-changing coefficient that {\it avoids} the $T$-coercivity
approach. The numerical method we present in this paper is always
well-posed {\it without} any conditions on the domain geometry or the
ratio between the negative and positive coefficients. This is achieved
by employing two key ideas for the problem: a mixed formulation of the
problem and using a hybridized discontinuous Galerkin (HDG) method
with {\it sign-changing} stabilization parameters. We will see that
the sign-changing stabilization parameter and the discontinuous test
function space in discontinuous Galerkin methods allow us to obtain
stable numerical methods without any non-quantitative assumptions on
mesh size or coefficients.

We also carry out an error analysis of the HDG method for this
problem. The error analysis of HDG methods typically uses the
Aubin--Nitsche duality argument with the full elliptic regularity
assumption. However, this assumption is not feasible in our setting,
because the PDE is not elliptic. By revealing a stabilized saddle
point problem structure of HDG methods, we circumvent this obstacle
and we can show an error estimate without using the duality
argument. We note that such analysis has been applied also to HDG
methods for standard Poisson-type problems to avoid the full elliptic
regularity assumption in the a priori error analysis (see
\cite{Cockburn:2014}).

The outline of this paper is as follows. In \cref{sec:problem} we
introduce the Poisson problem with a sign-changing coefficient. We
introduce the HDG method and discuss its well-posedness in
\cref{sec:hdg}. An a priori error analysis is presented in
\cref{sec:error_analysis}. In \cref{sec:superconvergence} we show that
a superconvergence result can be obtained when a suitable regularity
assumption is available. Our analysis is verified by numerical
examples in \cref{sec:num_examples}, while conclusions are drawn in
\cref{sec:conclusions}.

\section{The Poisson problem with a sign-changing coefficient}
\label{sec:problem}
Let $\Omega \subset \mathbb{R}^d$, $d \ge 2$, be a polygonal domain
that is divided into two subdomains $\Omega_+$ and $\Omega_-$ such
that
$\overline{\Omega} = \overline{\Omega}_+ \cup \overline{\Omega}_-$ and
$\Omega_+ \cap \Omega_- = \emptyset$. The boundaries of $\Omega$,
$\Omega_+$ and $\Omega_-$ are denoted by, respectively,
$\partial\Omega$, $\partial\Omega_+$ and $\partial\Omega_-$. The
interface separating the domains is denoted by
$\Gamma_I := \overline{\Omega}_+ \cap \overline{\Omega}_-$ and we
define furthermore $\Gamma_+ := \partial\Omega_+\backslash\Gamma_I$
and $\Gamma_- := \partial\Omega_-\backslash\Gamma_I$. We assume that
$\Gamma_D$ and $\Gamma_N$ are disjoint subsets of $\partial \Omega$
such that
$\partial \Omega = \overline{\Gamma_D} \cup \overline{\Gamma_N}$. The
outward unit normal vector field on $\partial\Omega$ is denoted by
$\nb$.

Let $\sigma$ be a scalar function defined as
\begin{equation}
  \sigma :=
  \begin{cases}
    \sigma_+ & \text{on}\ \Omega_+, \\
    \sigma_- & \text{on}\ \Omega_-,
  \end{cases}
\end{equation}
where $\sigma_+ > 0$ and $\sigma_- < 0$ are constants. The
\emph{contrast} is defined as $\kappa_{\sigma} := \sigma_-/\sigma_+$.
Throughout this paper we assume that 
\begin{align}
  \label{eq:sigma-min}
  0 < \sigma_{\min} \le  \sigma_+ ,\quad -\sigma_- \le \sigma_{\max} < +\infty.
\end{align}
We consider the following partial differential equation (PDE) for the
scalar $u : \Omega \to \mathbb{R}$:
\begin{subequations}
  \begin{align}
    \nabla \cdot \del{\sigma \nabla u} &= f && \text{ in } \Omega, \\
    u &= u_D && \text{ on } \Gamma_D, \\
    - \sigma \nabla u \cdot \nb &= u_N && \text{ on } \Gamma_N,
  \end{align}
  \label{eq:problem}
\end{subequations}
where $u_D : \Gamma_D \to \mathbb{R}$ and
$u_N : \Gamma_N \to \mathbb{R}$ are given boundary data and
$f : \Omega \to \mathbb{R}$ is a given source term. This is not a
standard second order PDE, because the sign of $\sigma$ is
indefinite. Let
$H^1_0(\Omega) := \cbr{v\in H^1(\Omega) \,:\, v=0 \text{ on }
  \partial\Omega}$ where $H^1(\Omega)$ is the standard Sobolev space
of $L^2$ functions such that the $L^2$-norm of their gradient is
bounded. The variational formulation for the problem (when
$u_D = u_N = 0$) is given by: Find $u \in H^1_0(\Omega)$ such that
\begin{equation}
  \label{eq:variationalForm}
  \int_{\Omega} \sigma \nabla u \cdot \nabla v \dif x
  = \int_{\Omega} f v \dif x \quad \forall v \in H^1_0(\Omega).
\end{equation}

It is known that this problem is not always well-posed, for example,
when $\kappa_{\sigma} = -1$, the problem is ill-posed.  The conditions
of well-posedness of the problem depends on the values of $\sigma$ and
the geometry of $\Omega_+$ and $\Omega_-$. For more information on the
well-posedness of this problem we refer to \cite{Chesnel:2013,
  Dhia:2012}.

\section{The HDG method}
\label{sec:hdg}
Introducing the auxiliary variable $\qb = -\sigma \nabla u$,
\cref{eq:problem} can be written as a system of first-order equations:
\begin{subequations}
  \begin{align}
    \sigma^{-1} \qb + \nabla u &= 0 && \text{ in } \Omega, \\
    \nabla \cdot \qb &= - f && \text{ in } \Omega, \\
    u &= u_D && \text{ on } \Gamma_D, \\
    - \sigma \nabla u \cdot \nb &= u_N && \text{ on } \Gamma_N.
  \end{align}
  \label{eq:problemmixed}
\end{subequations}
In this section we introduce the HDG method for \cref{eq:problemmixed}.

\subsection{Preliminaries}

Let $\mathcal{T}_h$ be a triangulation of $\Omega$ and $\mathcal{F}_h$
be the set of faces in the triangulation $\mathcal{T}_h$ with
co-dimension 1.  By $\mathcal{T}_h^+$ we denote the subset of
$\mathcal{T}_h$ such that its elements are in
$\Omega_+$. $\mathcal{T}_h^-$ is defined similarly. We assume that
$\mathcal{T}_h$ is a conforming triangulation with respect to
$\Gamma_I$, i.e., there is a subset $\mathcal{F}_h^i$ of
$\mathcal{F}_h$ such that
$\Gamma_I = \cup_{F \in \mathcal{F}_h^i} \overline{F}$. For later
reference we define $\mathcal{F}_h^+$ and $\mathcal{F}_h^-$ as the
subsets of interior facets of $\mathcal{F}_h$ such that the facets are
in $\Omega_+$ and $\Omega_-$, respectively. We also define
$\mathcal{F}_h^{\partial}$ as the facets on $\partial \Omega$, i.e.,
$\partial \Omega = \cup_{\mathcal{F}_h^{\partial}} \overline{F}$. In
summary, $\mathcal{F}_h$ is a disjoint union of
$\mathcal{F}_h^{\partial}$, $\mathcal{F}_h^+$, $\mathcal{F}_h^-$,
$\mathcal{F}_h^i$.  We also define $\mathcal{F}_h^D$ as the subset of
$\mathcal{F}_h$ such that
$\overline{\Gamma_D} = \cup_{F \in \mathcal{F}_h^D} \overline{F}$.

Let $K \in \mathcal{T}_h$, denote the diameter of $K$ by $h_K$, and
let $h = \max_{K\in\mathcal{T}_h}h_K$. We assume that there is no
$K \in \mathcal{T}_h$ such that $\partial K \subset \Gamma_I$. For
scalar functions $u, v \in L^2(K)$ and vector functions
$\qb, \rb \in L^2(K; \mathbb{R}^d)$ we denote, respectively,
$\del{u, v}_K := \int_K uv \dif x$ and
$\del{\qb, \rb}_K := \int_K \qb \cdot \rb \dif x$. Furthermore, for
two functions $\rb$ and $v$ with well-defined traces on $\partial K$,
we define
$\langle \rb \cdot \nb, v \rangle_{\partial K} := \int_{\partial K}
\rb \cdot \nb_K v \dif s$, where $\nb_K$ is the unit outward normal
vector field on $\partial K$. Additionally, we define
$\del{u, v}_{D} := \sum_{K\subset D} \del{u, v}_K$,
$\del{\qb, \rb}_{D} := \sum_{K\subset D} \del{\qb, \rb}_K$, and
$\langle \rb \cdot \nb, v \rangle_{\partial\mathcal{T}_h} := \sum_{K
  \in \mathcal{T}_h}\langle \rb \cdot \nb_K, v \rangle_{\partial K}$,
with similar definitions on $\mathcal{T}_h^+$ and $\mathcal{T}_h^-$.

We use the standard notation for Sobolev spaces based on the
$L^2$-norm, i.e., $H^s(D)$, $s \ge 0$, denotes the Sobolev space based
on the $L^2$-norm with $s$-differentiability on the domain $D$. We
refer to \cite{Evans:book} for a rigorous definition of this
space. The norm on $H^s(D)$ is denoted by $\norm{\cdot}_{s,D}$. When
$s=0$ we drop the subscript $s$.

To define the HDG method let 
\begin{align*}
  \Vb_h(K)  = \mathcal{P}_{k}(K; \mathbb{R}^d), \quad W_h(K) = \mathcal{P}_{k}(K), \quad  M_h(F) = \mathcal{P}_{k}(F),
\end{align*}
where $\mathcal{P}_k(D)$ and $\mathcal{P}_k(D;\mathbb{R}^d)$ are the
spaces of scalar and $d$-dimensional vector valued polynomials of
degree at most $k$ on a domain $D$. We will use the following finite
element spaces:
\begin{align}
  \Vb_h &= \cbr[0]{ \rb \in L^2(\mathcal{T}_h;\mathbb{R}^d) \;:\; \rb|_{K} \in \Vb_h(K), \quad \forall K \in \mathcal{T}_h }, \\
  W_h &= \cbr[0]{ v \in L^2(\mathcal{T}_h) \;:\; v|_{K} \in W_h(K), \quad \forall K \in \mathcal{T}_h }, \\
  M_{h} &= \cbr[0]{ \mu \in L^2(\mathcal{F}_h) \;:\; \mu|_{F} \in M_h(F), \quad \forall F \in \mathcal{F}_h, \quad \mu|_{\mathcal{F}_h^D} = 0 }.
\end{align}
%

\subsection{The discrete formulation}

For a sufficiently regular solution $(u, \boldsymbol{q})$ of the mixed
problem \cref{eq:problemmixed}, and for $\ubar$ the restriction of $u$
on $\mathcal{F}_h$, we can derive a system of variational equations
from \cref{eq:problemmixed} with test functions in
$\Vb_h \times W_h \times M_h$ as
\begin{subequations}
  \begin{align}
    \label{eq:exact-eq1}
    \del[0]{ \sigma^{-1} \qb, \rb }_{\Omega}
    + \del[0]{ \nabla u, \rb}_{\Omega}
    - \LRa{ u - \ubar, \rb \cdot \nb }_{\partial \mathcal{T}_h}
    &= 0  && \rb \in \Vb_h,
    \\
    \label{eq:exact-eq2}
    \del[0]{ \qb, \nabla v}_{\Omega}
    - \LRa{ \qb \cdot \nb + \tau (u - \ubar), v }_{\partial \mathcal{T}_h}
    &=  \del[0]{ f, v}_{\Omega} && v \in W_h,
    \\
    \label{eq:exact-eq3}
    \LRa{ \qb \cdot \nb + \tau (u - \ubar), \vbar}_{\partial \mathcal{T}_h \setminus \mathcal{F}_h^D}
    &= \LRa{ u_N, \vbar}_{ \mathcal{F}_h^N} 
         && \vbar \in M_{h} .
  \end{align}
  \label{eq:exact}
\end{subequations}
Here we use $\tau$ to denote a piecewise constant function adapted to
$\mathcal{F}_h$. Throughout this paper we assume that
\begin{align} 
  \label{eq:tau-condition}
  \tau|_F = 
  \begin{cases}
      0 & \text{ if } F \subset \Gamma_I, \\
    > 0 & \text{ if } F \subset \partial K, K \subset \Omega_+, F \not \subset \Gamma_I, \\
    < 0 & \text{ if } F \subset \partial K, K \subset \Omega_-, F \not \subset \Gamma_I.  
  \end{cases}
\end{align}
Our HDG method for \cref{eq:problemmixed} is the
discrete counterpart of this system of variational equations, i.e., we
find $(\qb_h, u_h, \ubar_h) \in \Vb_h \times W_h \times M_h$ such that
\begin{subequations}
  \begin{align}
    \label{eq:hdg-eq1}
    \del[0]{ \sigma^{-1} \qb_h, \rb }_{\Omega}
    + \del[0]{ \nabla u_h, \rb}_{\Omega}
    - \LRa{ u_h - \ubar_h, \rb \cdot \nb }_{\partial \mathcal{T}_h}
    &= - \LRa{u_D, \rb \cdot \nb}_{\mathcal{F}_h^D}  && \rb \in \Vb_h,
    \\
    \label{eq:hdg-eq2}
    \del[0]{ \qb_h, \nabla v}_{\Omega}
    - \LRa{ \qb_h \cdot \nb + \tau (u_h - \ubar_h), v }_{\partial \mathcal{T}_h}
    &= - \LRa{ \tau u_D, v}_{\mathcal{F}_h^D} + \del[0]{ f, v}_{\Omega} && v \in W_h,
    \\
    \label{eq:hdg-eq3}
    \LRa{ \qb_h \cdot \nb + \tau (u_h - \ubar_h), \vbar}_{\partial \mathcal{T}_h}
    &= \LRa{ u_N, \vbar}_{ \mathcal{F}_h^N} 
         && \vbar \in M_{h} .
  \end{align}
  \label{eq:hdgmethod}
\end{subequations}
Here equations \cref{eq:exact-eq1}, \cref{eq:exact-eq2} and the
discrete equations \cref{eq:hdg-eq1}, \cref{eq:hdg-eq2} look
inconsistent. This is because we impose the restriction
$\mu|_{\mathcal{F}_h^D} = 0$ on $M_h$ so that we can take the same
trial and test function spaces in our numerical method. This choice
will be useful to reveal a stabilized saddle point structure of our
numerical method, and also allow us to obtain optimal error estimates
without the Aubin--Nitsche duality argument.  To clarify the
consistency between \cref{eq:exact-eq1}, \cref{eq:exact-eq2} and
\cref{eq:hdg-eq1}, \cref{eq:hdg-eq2}, we point out that
\cref{eq:hdg-eq1}, \cref{eq:hdg-eq2} with $\bar{u}_h \in \tilde{M}_h$
are equivalent to the discretized forms of \cref{eq:exact-eq1},
\cref{eq:exact-eq2} with $\tilde{M}_h$ defined by
\begin{multline*}
  \tilde{M}_{h} = \\
  \cbr[0]{ \mu \in L^2(\mathcal{F}_h)\, : \, \mu|_{F} \in M_h(F)
    \ \forall F \in \mathcal{F}_h,
    \ \langle \mu, \lambda \rangle_{\mathcal{F}_h^D} = \langle u_D, \lambda \rangle_{\mathcal{F}_h^D}
    \ \forall \lambda \in M_h(\mathcal{F}_h^D) } .
\end{multline*}

For later use we define
\begin{equation}
  \begin{split}
    a_h(\qb, \rb)
    &:= \del[0]{\sigma^{-1} \qb, \rb}_{\Omega}, 
    \\
    b_h(v, \vbar; \rb)
    &:=
    \del[0]{\nabla v, \rb}_{\Omega}
    - \LRa{ v - \vbar, \rb\cdot \nb}_{\partial \mathcal{T}_h},
    \\
    c_h(u, \ubar; v, \vbar)
    &:=
    \LRa{ \tau (u - \ubar), v - \vbar }_{\partial \mathcal{T}_h}.
  \end{split}
\end{equation}
Then the sum of the left-hand sides of \cref{eq:hdgmethod} can be
rewritten as
\begin{equation}
  \label{eq:ah-form}
  B_h(\qb_h, u_h, \ubar_h; \rb, v, \vbar) := a_h(\qb_h, \rb) + b_h(u_h, \ubar_h; \rb)
  + b_h(v, \vbar; \qb_h) - c_h(u_h, \ubar_h; v, \vbar).
\end{equation}

For brevity, but without loss of generality, we assume $u_D = u_N = 0$
and $\Gamma_D \ne \emptyset$ in the remainder of this paper.

In the following lemma we prove well-posedness of the HDG method
\cref{eq:hdgmethod}.

\begin{lemma}[Well-posedness]
  Let $k\ge 0$. If $\tau$ satisfies \cref{eq:tau-condition}, then
  there exists a unique solution
  $(\qb_h, u_h, \ubar_h) \in \Vb_h \times W_h \times M_h$ to
  \cref{eq:hdgmethod}.
\end{lemma}
\begin{proof}
  To show well-posedness of \cref{eq:hdgmethod} assume that $f =
  0$. We will show that $\qb_h$, $u_h$, $\ubar_h$ vanish.
  
  Let $\chi_{D}$ be the characteristic function that has the value 1
  in domain $D$ and 0 elsewhere. Then let
  $\rb = \chi_{\Omega_+} \qb_h - \chi_{\Omega_-} \qb_h$,
  $v = - \chi_{\Omega_+} u_h + \chi_{\Omega_-} u_h$,
  $\vbar = \chi_{\overline{\Omega_+}} \ubar_h -
  \chi_{\overline{\Omega_-}} \ubar_h$ in \cref{eq:hdgmethod}, and add
  all the equations. (Note that $\vbar = 0$ on $\Gamma_I$.) The
  evaluation of \cref{eq:hdgmethod} with the first components of these
  test functions gives
  \begin{align*}
    \del[0]{ \sigma^{-1} \qb_h, \qb_h}_{\Omega_+}
    + \LRa{ \tau (u_h - \ubar_h), u_h - \ubar_h}_{\partial \mathcal{T}_h^+ }
    = 0. 
  \end{align*}
  Similarly, the evaluation of \cref{eq:hdgmethod} with 
  the second components of the above test functions gives 
  \begin{align*}
    - \del[0]{\sigma^{-1} \qb_h, \qb_h}_{\Omega_-}
    - \LRa{ \tau (u_h - \ubar_h), u_h - \ubar_h}_{\partial \mathcal{T}_h^- }
    = 0. 
  \end{align*}
  Since $\sigma^{-1}$ and $\tau$ are positive and negative on
  $\Omega_+$ and $\Omega_-$, respectively, we can conclude that
  $\qb_h = 0$ on $\Omega$ and $u_h = \ubar_h$ on
  $\mathcal{F}_h^+ \cup \mathcal{F}_h^- \cup
  \mathcal{F}_h^{\partial}$. This also implies that $u_h$ is
  continuous on $\Omega_+$ and on $\Omega_-$, respectively. It then
  follows from \cref{eq:hdg-eq1} that $\nabla u_h = 0$, i.e., $u_h$ is
  constant on $\Omega_+$ and on $\Omega_-$, respectively. Since
  $\qb_h = 0$, $\nabla u_h = 0$, and $u_D = 0$, \cref{eq:hdg-eq1}
  gives
  $\lra{ u_h - \ubar_h, \rb \cdot \nb }_{\partial \mathcal{T}_h}=0$
  for all $\rb \in \Vb_h$, implying $u_h = \ubar_h$ also on
  $\Gamma_I$. Given now that $u_h$ is constant on $\Omega$, we obtain
  that $u_h = \ubar_h = 0$ because $u_h = \ubar_h = 0$ on $\Gamma_D$.
  %
  %
\end{proof}

We consider the following norm and seminorm on $\Vb_h$ and
$W_h \times M_h$:
\begin{align*}
  \norm[0]{ \rb }_{\Vb_h}^2 &= \del[0]{\sigma^{-1} \rb, \rb}_{\Omega_+} - \del[0]{\sigma^{-1} \rb, \rb}_{\Omega_-}, \\
  \seminorm{ (v, \vbar) }_{W_h \times M_h}^2 &=  \del[0]{ v , v}_{\Omega} 
  + \LRa{ |\tau| \del[0]{v - \vbar}, v - \vbar}_{\partial \mathcal{T}_h},
\end{align*}
where $|\tau|$ is the absolute value of $\tau$. Additionally, we
define $\Xb_h := \Vb_h \times W_h \times M_h$ with seminorm
\begin{align*}
  \seminorm{(\rb, v, \vbar)}_{\Xb_h} := \del[1]{\, \norm[0]{\rb}_{\Vb_h}^2 + \seminorm{(v, \vbar)}_{W_h \times M_h}^2 }^{\frac 12} .
\end{align*}
Furthermore, in the remainder of this paper we let $C>0$ be a constant
independent of the mesh size $h$.

\begin{lemma} 
  \label{lemma:div-surjective}
  There exists $C_{\Omega} >0$ which depends only on $\Omega$ and
  $\Gamma_D$ such that for $v \in W_h$ there exists
  $\wb \in H^1(\Omega; \Bbb{R}^d)$ satisfying $\nabla \cdot \wb = v$
  and $\norm{ \wb }_{1,\Omega} \leq C_{\Omega} \norm{ v }_{\Omega}$
  with $\wb \cdot \nb = 0$ on $\Gamma_N$.
\end{lemma}
\begin{proof}
  We will use a result in \cite[p.176]{Galdi:2011} which claims that
  for any given $f \in L^2(\Omega)$ and
  $\boldsymbol{a} \in H^{\frac 12}(\partial \Omega)$ satisfying
  $\int_{\partial \Omega} \boldsymbol{a} \cdot \nb \dif s =
  \int_{\Omega} f \dif x$, there exists
  $\wb \in H^1(\Omega; \Bbb{R}^d)$ such that $\nabla \cdot \wb = f$,
  $\wb = \boldsymbol{a}$ on $\partial \Omega$, and
  $\norm{ \wb }_{H^1(\Omega)} \leq C \norm{ f }_{L^2(\Omega)} +
  \norm{ \boldsymbol{a} }_{H^{\frac 12}(\partial \Omega)}$.

  We first note that if $z \in L^2(\Omega)$ has mean value zero, then
  there exists $\wb \in H^1_0(\Omega; \mathbb{R}^d)$ such that
  $\nabla \cdot \wb = z$ and
  $\norm{\wb}_{H^1(\Omega)} \le C \norm{z}_{L^2(\Omega)}$ by taking
  $\boldsymbol{a} = \boldsymbol{0}$ in the above result.

  Consider now a decomposition $v = v_0 + v_m$ with
  $v_m = \chi_{\Omega} \frac{1}{|\Omega|}\int_{\Omega} v \dif x$ where
  $|\Omega|$ is the Lebesgue measure of $\Omega$, and where $v_0$ is
  the component of $v$ with mean-value zero. It is clear that
  $\norm{ v_0 }_{L^2(\Omega)} + \norm{ v_m }_{L^2(\Omega)} \leq C
  \norm{ v }_{L^2(\Omega)}$. As just noted, there exists
  $\wb_0 \in H_0^1(\Omega; \Bbb{R}^d)$ such that
  $\nabla \cdot \wb_0 = v_0$ and
  $\norm{ \wb_0 }_{H^1(\Omega)} \leq C \norm{ v_0 }_{L^2(\Omega)}$.
  Since $\Omega$ is a polygonal/polyhedral domain, one can find a
  function $\boldsymbol{\phi} \in H^{1/2}(\partial \Omega)$ whose
  support resides on $\Gamma_D$,
  $\int_{\Gamma_D} \boldsymbol{\phi} \cdot \nb \dif s = c_0 > 0$, and
  $\norm{ \boldsymbol{\phi} }_{H^{\frac 12}(\partial \Omega)} = 1$.
  We can find $\boldsymbol{\phi}_0$, a renormalization of
  $\boldsymbol{\phi}$, satisfying
  $\int_{\Gamma_D} \boldsymbol{\phi}_0 \cdot \nb \dif s =
  \int_{\Omega} v_m \dif x$, and it holds that
  \begin{equation*}
    \norm{ \boldsymbol{\phi}_0 }_{H^{\frac 12}(\partial \Omega)}
    \le
    \frac{\norm{ v_m }_{L^1(\Omega)}}{c_0}
    \leq C
    \norm{ v_m }_{L^2(\Omega)},
  \end{equation*}
  with a constant that depends on $|\Omega|$ and $c_0$. Applying the
  result in \cite[p.176]{Galdi:2011} there exists
  $\wb_1 \in H^1(\Omega; \Bbb{R}^d)$ such that
  $\nabla \cdot \wb_1 = v_m$, $\wb_1 = \boldsymbol{\phi}_0$ on
  $\partial \Omega$, and
  $\norm{ \wb_1 }_{H^1(\Omega)} \leq C
  \norm{\boldsymbol{\phi}_0}_{H^{1/2}(\partial \Omega)} + \norm{ v_m
  }_{L^2(\Omega)} \leq C \norm{ v_m }_{L^2(\Omega)}$. Then
  $\wb = \wb_0 + \wb_1$ is a function in $H^1(\Omega; \Bbb{R}^d)$
  satisfying the desired conditions.
\end{proof}

Before we prove an inf-sup condition, that we use in the error
analysis in \cref{sec:error_analysis}, let us recall a result from
\citep[Proposition~2.1]{Cockburn-Dong:2008}: For a simplex $K$ and a
fixed facet $F_K$ of $K$, there is an interpolation operator
$\Pi_h' : H^1(K; \mathbb{R}^d) \rightarrow \Vb_h(K)$ such that
\begin{align*}
  \del[0]{\Pi_h' \rb, \rb'}_K &= \del[0]{ \rb, \rb'}_K & & \forall \rb' \in \mathcal{P}_{k-1}(K; \mathbb{R}^d) \quad \text{ if } k > 0, \\
  \lra{ \Pi_h' \rb \cdot \nb, w }_F &= \lra{ \rb \cdot \nb, w }_F & & \forall w \in \mathcal{P}_k (F) \quad \text{ if } F \not = F_K.
\end{align*}
This interpolation operator $\Pi_h'$ also satisfies
$\norm[0]{ (\rb - \Pi_h' \rb) \cdot \nb }_{\partial K} \le C h_K^{1/2}
\norm[0]{\rb}_{H^1(K)}$.

\begin{theorem}[inf-sup condition]
  \label{thm:infsupcondition}
  Suppose that $k \geq 0$.  If $\tau$ satisfies
  \cref{eq:tau-condition} and $\gamma_0 \le |\tau| \le \gamma_1$ for
  $\tau \not = 0$ on $\mathcal{F}_h$ with positive constants
  $\gamma_0$ and $\gamma_1$ independent of $h$, then there exists a
  positive constant $\beta$ independent of $h$ such that
  \begin{align}
    \label{eq:Bh-inf-sup}
    \inf_{(\pb, z, \zbar) \in \Xb_h } \sup_{(\rb, v, \vbar) \in \Xb_h }
    \frac{B_h(\pb, z, \zbar; \rb, v, \vbar)}{\seminorm{ (\pb, z, \zbar) }_{\Xb_h} \seminorm{ (\rb, v, \vbar) }_{\Xb_h} } \ge \beta >0.
  \end{align}
  Here we assume that $\seminorm{ (\pb, z, \zbar) }_{\Xb_h}$,
  $\seminorm{ (\rb, v, \vbar) }_{\Xb_h} > 0$.
\end{theorem}
\begin{proof}
  We first prove the following weak inf-sup condition: There exist
  $C_0, C_1 >0$ independent of $h$ such that
  \begin{equation}
    \label{eq:weak_inf_sup}
    \inf_{(z, \zbar) \in W_h \times M_h} \sup_{\rb \in \Vb_h }
    \frac{b_h(z, \zbar; \rb)}{\norm{ \rb }_{\Vb_h} }
    \ge C_0 \norm{ z }_{\Omega} - C_1 \LRa{ |\tau| (z - \zbar), z - \zbar }_{\partial \mathcal{T}_h}^{\frac 12} .
  \end{equation}
  We show this by proving an equivalent condition, i.e., there exist
  $C_0', C_1', C_2' >0$ independent of $h$ such that for any
  $0 \neq (z, \zbar) \in W_h \times M_h$ there exists
  $\rb \in \Vb_h$ such that
  $\norm{ \rb }_{\Vb_h} \le C_2' \norm{ z }_{\Omega}$ and
  \begin{align}
    \label{eq:weak-inf-sup}
    b_h(z, \zbar; \rb)
    \ge C_0' \norm{ z }_{\Omega}^2
    - C_1' \LRa{ |\tau| (z - \zbar), z - \zbar }_{\partial \mathcal{T}_h}^{\frac 12} \norm{ z }_{\Omega}.
  \end{align}

  By \cref{lemma:div-surjective} there exists
  $\wb \in H^1(\Omega; \Bbb{R}^d)$ with
  $\wb \cdot \nb|_{\Gamma_N} = 0$ such that $\nabla \cdot \wb = -z$
  and $\norm{ \wb }_{1,\Omega} \le C_{\Omega} \norm{ z }_{\Omega}$
  with $C_{\Omega}$ depending only on $\Omega$ and $\Gamma_D$. We
  define $\Pi_h \wb$ on $K$ as follows: If
  $\partial K \cap \Gamma_I = \emptyset$, then $\Pi_h \wb$ is the
  $L^2$ orthogonal projection of $\wb$ into $\Vb_h$. If
  $\partial K \cap \Gamma_I \not = \emptyset$, then
  $\Pi_h \wb = \Pi_h' \wb$ with the interpolation $\Pi_h'$ associated
  to a facet $F \subset \partial K$, $F \not \subset \Gamma_I$. It
  then holds that
  \begin{align}
    \notag \del[0]{ \wb , \rb}_K &= \del[0]{ \Pi_h \wb , \rb}_K && \forall \rb \in \mathcal{P}_{k-1}(K; \mathbb{R}^d) \text{ for all } K \in \mathcal{T}_h,
    \\
    \label{eq:intp-normal}
    \lra{ \wb \cdot \nb, w }_F &= \lra{ \Pi_h \wb \cdot \nb, w }_F && \forall w \in \mathcal{P}_{k}(F) \text{ for all } F \in \mathcal{F}_h^i.
  \end{align}
  Then
  \begin{equation}
    \label{eq:bhuhw}
    \begin{split}
      -\norm{ z }_{\Omega}^2 
      &= \del[0]{z, \nabla \cdot \wb}_{\Omega} \\
      &= - \del[0]{ \nabla z, \wb}_{\Omega} + \LRa{ z, \wb \cdot \nb }_{\partial \mathcal{T}_h} \\
      &= - \del[0]{ \nabla z, \Pi_h \wb}_{\Omega} + \LRa{ z - \zbar, \wb \cdot \nb }_{\partial \mathcal{T}_h}
      \\
      &= - b_h\del[0]{z, \zbar; \Pi_h \wb} + \LRa{ z - \zbar, (\wb - \Pi_h \wb ) \cdot \nb }_{\partial \mathcal{T}_h \setminus \Gamma_I}
      ,     
    \end{split}
  \end{equation}
  where the third equality holds because $\wb\cdot\nb$ and $\zbar$
  are single-valued on facets, $\zbar = 0$ on $\Gamma_D$ by the
  definition of $M_h$, and $\wb \cdot \nb = 0$ on $\Gamma_N$, and the fourth equality holds because of \cref{eq:intp-normal}.
    
  The element-wise trace inequality and the Cauchy-Schwarz inequality
  give
  \begin{equation}
    \label{eq:traceineqangles}
    \begin{split}
     \LRa{ z - \zbar, (\wb - \Pi_h \wb ) \cdot \nb }_{\partial \mathcal{T}_h  \setminus \Gamma_I}
     & \ge - C \LRa{ |\tau| (z - \zbar), z - \zbar }_{\partial \mathcal{T}_h}^{\frac 12} \norm{ \wb }_{1,\Omega}
     \\
     & \ge - C' \LRa{ |\tau| (z - \zbar), z - \zbar }_{\partial \mathcal{T}_h}^{\frac 12} \norm{ z }_{\Omega}.      
    \end{split}
  \end{equation}
  Combining \cref{eq:bhuhw} and \cref{eq:traceineqangles} we obtain
  \begin{equation}
    b_h\del[0]{z, \zbar; \Pi_h \wb} \ge \norm{ z }_{\Omega}^2
    - C' \LRa{ |\tau| (z - \zbar), z - \zbar }_{\partial \mathcal{T}_h}^{\frac 12} \norm{ z }_{\Omega}.
  \end{equation}
  Moreover, 
  \begin{equation}
    \norm{ \Pi_h \wb }_{\Omega}
    \le C \norm{ \wb }_{1,\Omega}
    \le C_{\Omega} \norm{ z }_{\Omega}
  \end{equation}
  holds, so the weak inf-sup condition \cref{eq:weak-inf-sup} follows.
  
  To prove \cref{eq:Bh-inf-sup} suppose that $\rb_0 \in \Vb_h$ is an
  element satisfying \cref{eq:weak-inf-sup} for given $z$. For given
  $(\pb, z, \bar{z}) \in \Xb_h$ we take
  $\rb = \chi_{\Omega_+} \pb - \chi_{\Omega_-} \pb + \varepsilon \rb_0
  $, $v = - \chi_{\Omega_+} z + \chi_{\Omega_-} z$,
  $\vbar = -\chi_{\overline{\Omega_+}} \zbar +
  \chi_{\overline{\Omega_-}} \zbar$ in \cref{eq:ah-form}, with
  $\varepsilon > 0$ to be determined later. Then by
  \cref{eq:weak-inf-sup} and Young's inequality,
  \begin{align*}
    B_h(\pb, z, \zbar; \rb, v, \vbar) 
    &= \norm{ \pb }_{\Vb_h}^2 + \varepsilon \del[0]{\sigma^{-1} \pb, \rb_0}_{\Omega}  + \varepsilon b_h(z, \zbar; \rb_0) \\
    &\quad + \LRa{ |\tau| (z - \zbar), z - \zbar}_{\partial \mathcal{T}_h}
    \\
    &\ge \norm{ \pb }_{\Vb_h}^2 - \frac 12 \norm{ \pb}_{\Vb_h}^2 - \frac 12 \varepsilon ^2 (C_2')^2 \norm{ z }_{\Omega}^2
    \\
    &\quad + \varepsilon C_0' \norm{ z }_{\Omega}^2 
    -  \varepsilon C_1' \LRa{ |\tau| (z - \zbar), z - \zbar }_{\partial \mathcal{T}_h}^{\frac 12} \norm{ z }_{\Omega} 
    \\
    &\quad + \LRa{ |\tau| (z - \zbar), z - \zbar}_{\partial \mathcal{T}_h} \\
    &\ge \norm{ \pb }_{\Vb_h}^2 - \frac 12 \norm{ \pb}_{\Vb_h}^2 - \frac 12  \varepsilon ^2 (C_2')^2 \norm{ z }_{\Omega}^2
    \\
    &\quad + \varepsilon C_0' \norm{ z }_{\Omega}^2 
    -  \frac 12 \LRa{ |\tau| (z - \zbar), z - \zbar }_{\partial \mathcal{T}_h} \\
    &\quad - \frac 12 \varepsilon^2 (C_1')^2  \norm{ z }_{\Omega}^2 
    + \LRa{ |\tau| (z - \zbar), z - \zbar}_{\partial \mathcal{T}_h}.     
  \end{align*}
  Choosing $\varepsilon$ sufficiently small, we obtain
  \begin{equation*}
    B_h(\pb, z, \zbar; \rb, v, \vbar) 
    \ge
    C \del[1]{ \norm{ \pb }_{\Vb_h}^2
      + \norm{ z }_{\Omega}^2
      + \LRa{ |\tau| (z - \zbar), z - \zbar }_{\partial \mathcal{T}_h} }.
  \end{equation*}
  Finally, it is not difficult to show, by the choice of
  $(\rb, v, \bar{v})$, that
  \begin{align*}
    \seminorm{ (\rb, v, \vbar) }_{\Xb_h}
    \leq C \seminorm{ ( \pb, z, \zbar) }_{\Xb_h}.
  \end{align*}
  This proves \cref{eq:Bh-inf-sup}.
\end{proof}

\section{Error analysis}
\label{sec:error_analysis}
In this section we present the \emph{a priori} error estimates of our HDG
method. For this we recall an interpolation result
(cf. \cite[Theorem~2.1]{Cockburn:2010}): For $K \in \mathcal{T}_h$,
$\tau |_{\partial K}$ is a piecewise constant function
which is nonnegative if $K \subset \Omega_+$ and nonpositive if $K \subset \Omega_-$ 
such that $\tau |_{\partial K} \not = 0$. Then there exists a bounded
interpolation operator $\boldsymbol{\Pi} := (\Pi_{\Vb}, \Pi_W)$ from
$ H^1(K; \Bbb{R}^d) \times H^1(K)$ to $\Vb_h(K) \times W_h(K)$ such
that
\begin{subequations}
  \begin{align}
    \label{eq:intp1}
    \del[0]{\Pi_{\Vb} \rb, \tilde{\rb}}_K &= \del[0]{ \rb, \tilde{\rb} }_K
    & &  \forall \tilde{\rb} \in \sbr[0]{\mathcal{P}_{k-1}(K) }^d , \\
    \label{eq:intp2}
    \del[0]{\Pi_W v, \tilde{v}}_K &= \del[0]{ v, \tilde{v}}_K
    & &  \forall \tilde{v} \in \mathcal{P}_{k-1}(K), \\
    \label{eq:intp3}
    \LRa{ \Pi_{\Vb} \rb \cdot \nb + \tau  \Pi_W v, \lambda}_{F} &= \LRa{ \rb \cdot \nb + \tau v, \lambda}_{F}
    & &  \forall \lambda \in \mathcal{P}_k (F),
  \end{align}
  \label{eq:intp}
\end{subequations}
for $F \subset \partial K$ and
$(\rb, v) \in H^1(K; \Bbb{R}^d) \times H^1(K)$, and
\begin{align}
  \label{eq:interpolationestimate-1}
\norm{\rb - \Pi_{\Vb} \rb}_{K} &\leq C h_K^{k_{\rb} +1} | \rb |_{k_{\rb}+1, K} + h_K^{k_v +1} \tau_K^* | v |_{k_v + 1, K}, \\
  \label{eq:interpolationestimate-2}
\norm{v - \Pi_W v}_{K} &\leq C h_K^{k_v + 1} | v |_{k_v +1, K} + \frac{h_K^{k_{\rb}+1}}{\tau_{K}^{\max}} | \nabla \cdot \rb |_{k_{\rb}, K},
\end{align}
hold with $0 \le k_v, k_{\rb} \le k$. Here,
\begin{align*}
  &\tau_K^{\max} = 
  \begin{cases}
    \max_{F \subset \partial K} \tau|_F  &\text{ if } \tau \ge 0, \\
    \max_{F \subset \partial K} (- \tau|_F)  &\text{ if } \tau \le 0 ,
  \end{cases} \\
  &\tau_K^* = 
  \begin{cases}
    \max_{F \subset \partial K \setminus F^*} \tau|_F  &\text{ if } \tau \ge 0, \\
    \max_{F \subset \partial K \setminus F^*} (- \tau|_F)  &\text{ if } \tau \le 0 ,
  \end{cases}
\end{align*}
where $F^*$ is the face where $\tau_K^{\max}$ is attained, and the
implicit constants are independent of $K$ and $\tau$. We remark that
the proof in \cite{Cockburn:2010} is only for nonnegative $\tau$. The
proof, however, can easily be adapted to nonpositive $\tau$, so we
omit its proof here.  We will also require $P_M$, the facet-wise $L^2$
projection from $L^2(\mathcal{F}_h \setminus \mathcal{F}_h^D)$ to
$M_h$.

For brevity of notation, we will denote the difference between an
unknown $\phi$ and its approximation $\phi_h$ by
$e_{\phi} := \phi - \phi_h$. It will be convenient to also split the
error into interpolation and approximation errors:
\begin{align}
  e_{\qb} &= e_{\qb}^I + e_{\qb}^h := (\qb - \Pi_{\Vb} \qb) + (\Pi_{\Vb} \qb - \qb_h) , \notag \\
  e_u &= e_u^I + e_u^h :=  (u - \Pi_W u) + (\Pi_W u - u_h), \label{eq:error-decomposition} \\
  e_{\ubar} &= e_{\ubar}^I + e_{\ubar}^h := (\ubar - P_M \ubar) + (P_M \ubar - \ubar_h) . \notag 
\end{align}
The error equations are obtained by subtracting \cref{eq:hdgmethod}
from \cref{eq:exact} and integration by parts,
\begin{subequations}
  \begin{align}
    \del[0]{ \sigma^{-1} e_{\qb}, \rb }_{\Omega}
    - \del[0]{ e_u,  \nabla \cdot \rb}_{\Omega}
    + \LRa{e_{\ubar}, \rb \cdot \nb }_{\partial \mathcal{T}_h}
    &=0 && \rb \in \Vb_h,
    \\
    \del[0]{ e_{\qb}, \nabla v}_{\Omega}
    - \LRa{ e_{\qb} \cdot \nb + \tau (e_{u} - e_{\ubar}), v }_{\partial \mathcal{T}_h}
    &=0 && v \in W_h,
    \\
    \LRa{ e_{\qb} \cdot \nb + \tau (e_{u} - e_{\ubar}), \vbar}_{\partial \mathcal{T}_h \setminus \mathcal{F}_h^D}
    &=0
         && \vbar \in M_{h} .
  \end{align}
  \label{eq:error_eqns}
\end{subequations}

\begin{theorem}[A priori error estimates]
  Let $(\qb, u) \in H^s(\Omega; \Bbb{R}^d) \times H^s(\Omega)$,
  $s \ge 1$ be the solution of \cref{eq:exact} and let $\ubar$ be the
  trace of $u$ on $\mathcal{F}_h \backslash \mathcal{F}_h^D$ and
  $\ubar = 0 $ on $\mathcal{F}_h^D$. We assume that $\tau$ satisfies
  the same conditions as in \cref{thm:infsupcondition}, and let
  $(\qb_h, u_h, \ubar_h) \in \Vb_h \times W_h \times M_h$ solve
  \cref{eq:hdgmethod}. Then,
  \begin{subequations}
    \begin{align}
      \label{eq:q-err-estm}
      \norm{\qb - \qb_h}_{\Vb_h}  &\leq C h^s (\norm{\qb}_{H^s(\Omega)} + \norm{u}_{H^s(\Omega)} ),
      && 1 \le s \le k+1,
      \\
      \label{eq:u-err-estm}
      \norm{u - u_h}_{\Omega} &\leq C h^s (\norm{\qb}_{H^s(\Omega)} + \norm{u}_{H^s(\Omega)} ),
      && 1 \le s \le k+1,
      \\
      \label{eq:ubar-err-estm}
      \LRa{ P_M u - \ubar_h, P_M u - \ubar_h}^{\frac 12}_{\partial\mathcal{T}_h \backslash \mathcal{F}_h^D}
                                  &\leq C h^{s-\frac 12} (\norm{\qb}_{H^s(\Omega)} + \norm{u}_{H^s(\Omega)} ),
      && 1 \le s \le k+1.
    \end{align}    
  \end{subequations}  
\end{theorem}
\begin{proof}
  By definition of the interpolation operator \cref{eq:intp}, the
  error equations \cref{eq:error_eqns} can be reduced to
  \begin{subequations}
    \begin{align}
      \label{eq:red-err-eq1}
      \del[0]{ \sigma^{-1} e_{\qb}^h, \rb }_{\Omega}
      - \del[0]{ e_u^h,  \nabla \cdot \rb}_{\Omega}
      + \lra{e_{\ubar}^h, \rb \cdot \nb }_{\partial \mathcal{T}_h}
      &= -\del[0]{ \sigma^{-1} e_{\qb}^I, \rb }_{\Omega}  && \rb \in \Vb_h,
      \\
      \label{eq:red-err-eq2}
      \del[0]{ e_{\qb}^h, \nabla v}_{\Omega}
      - \lra{ e_{\qb}^h \cdot \nb + \tau (e_{u}^h - e_{\ubar}^h), v }_{\partial \mathcal{T}_h}
      &=0 && v \in W_h,
      \\
      \label{eq:red-err-eq3}
      \lra{ e_{\qb}^h \cdot \nb + \tau (e_{u}^h - e_{\ubar}^h), \vbar}_{\partial \mathcal{T}_h \setminus \mathcal{F}_h^D}
      &=0  && \vbar \in M_{h} .
    \end{align}
    \label{eq:red-err-eq}
  \end{subequations}
  To show  \cref{eq:q-err-estm}, take 
  $\boldsymbol{r} = \chi_{\Omega_+} e_{\qb}^h  - \chi_{\Omega_-} e_{\qb}^h$, 
  $v = - \chi_{\Omega_+} e_u^h + \chi_{\Omega_-} e_u^h$,
  $\vbar = \chi_{\overline{\Omega_+}} e_{\ubar}^h -
  \chi_{\overline{\Omega_-}} e_{\ubar}^h$ in \cref{eq:red-err-eq}, 
  and add all the equations. We find
  \begin{align*}
    \del[0]{ |\sigma|^{-1} e_{\qb}^h, e_{\qb}^h }_{\Omega} 
    + \lra{ |\tau| \del[0] {e_u^h - e_{\ubar}^h}, e_u^h - e_{\ubar}^h }_{\partial \mathcal{T}_h}
    &= - \del[0]{ \sigma^{-1} e_{\qb}^I, \rb }_{\Omega}
    \\
    &\le \del[0]{ |\sigma|^{-1} e_{\qb}^I, e_{\qb}^I }_{\Omega}^{\frac 12} \del[0]{ |\sigma|^{-1} e_{\qb}^h, e_{\qb}^h }_{\Omega}^{\frac 12} . 
  \end{align*}
  By this inequality, the triangle inequality, and \cref{eq:interpolationestimate-1},
  one can obtain \cref{eq:q-err-estm}. 
  
  To show \cref{eq:u-err-estm}, we note a result from
  \cref{thm:infsupcondition}: There exists $(\rb, v, \vbar) \in \Xb_h$
  such that
  $B_h(e_{\qb}^h, e_u^h, e_{\ubar}^h; \rb, v, \vbar) \ge
  \envert[0]{(e_{\qb}^h, e_u^h, e_{\ubar}^h)}_{\Xb_h}$ with
  $\seminorm{\rb, v, \vbar}_{\Xb_h} \leq C$. Together with
  \cref{eq:red-err-eq} and \cref{eq:interpolationestimate-1} this
  results in
  \begin{equation*}
    \envert[0]{(e_{\qb}^h, e_u^h, e_{\ubar}^h)}_{\Xb_h}
    \leq C
    \norm[0]{e_{\qb}^I}_{\Vb_h}
    \leq C h^s (\norm{\qb}_{H^s(\Omega)} + \norm{u}_{H^s(\Omega)}), \qquad 1 \le s \le k+1 .
  \end{equation*}
  Then \cref{eq:u-err-estm} follows by the triangle inequality.

  We next show \cref{eq:ubar-err-estm}. For each $K \in \mathcal{T}_h$
  there exists $\rb \in \Vb_h (K)$ such that
  $\rb \cdot \nb = e_{\ubar}^h$ on $\partial K$ and
  $\norm[0]{ \rb }_K \le C h_K^{\frac 12}
  \norm[0]{e_{\ubar}^h}_{\partial K}$ with $C$ independent of
  $K$. Taking this $\rb$ in \cref{eq:red-err-eq1}, we obtain
  \begin{align*}
    \lra{ e_{\ubar}^h, e_{\ubar}^h }_{\partial \mathcal{T}_h} 
    &=  \envert[1]{ - \del[0]{ \sigma^{-1} e_{\qb}, \rb }_{\Omega}  
      - \del[0]{ e_u^h,  \nabla \cdot \rb}_{\Omega} }
    \\
    &\le  C \norm[0] { e_{\qb} }_{\Omega} \norm[0]{ \rb }_{\Omega}  
      + \norm[0]{ e_u^h }_{\Omega} \norm{  \nabla \cdot \rb}_{\Omega}
    \\
    &\le C \norm[0]{e_{\qb} }_{\Omega} h^{\frac 12}
         \lra{ e_{\ubar}^h, e_{\ubar}^h }_{\partial \mathcal{T}_h}^{\frac 12} 
         + C h^{-\frac 12} \norm[0]{e_u^h}_{\Omega} 
         \lra{ e_{\ubar}^h, e_{\ubar}^h }_{\partial \mathcal{T}_h}^{\frac 12}
  \end{align*}
  where we used an inverse inequality and the properties of the chosen
  $\rb$ in the last inequality.  Then \cref{eq:ubar-err-estm} follows
  by \cref{eq:q-err-estm} and \cref{eq:u-err-estm}.
\end{proof}

\begin{remark}
  We remark that the proof of \cref{eq:q-err-estm} is not
  new. However, to the best of our knowledge, the proofs of the
  estimates \cref{eq:u-err-estm} and \cref{eq:ubar-err-estm} for
  sign-changing stabilization parameters without a duality argument
  have not been reported in the literature.
\end{remark}

\section{Superconvergence by a duality argument}
\label{sec:superconvergence}
In this section we discuss superconvergence of $u_h$ by a duality
argument. For $\theta \in L^2(\Omega)$, let us consider a dual
problem,
\begin{subequations}
  \begin{align}
    \label{eq:dualproblem_a}
    \sigma^{-1} \Phib + \nabla \phi &= 0 && \text{ in } \Omega,
    \\
    \nabla \cdot \Phib &= \theta  && \text{ in } \Omega,
    \\
    \phi &= 0 && \text{ on } \partial \Omega.
  \end{align}
  \label{eq:dualproblem}
\end{subequations}
We assume that the solution $(\Phib, \phi)$ satisfies the regularity
assumption
\begin{equation}
  \label{eq:regularity}
  \norm{\Phib}_{1, \Omega_+} + \norm{\Phib}_{1, \Omega_-}
  + \norm{\phi}_{2, \Omega_+} + \norm{\phi}_{2, \Omega_-}
  \le C_{\rm reg} \norm{ \theta }_{\Omega}.
\end{equation}
This assumption holds, for example, when both $\Omega_+$ and
$\Omega_-$ are convex (cf. \cite[Proposition~2]{Chesnel:2013}).

The following lemma is an analogue of Lemma~4.2 in
\cite{Cockburn:2010}.  The proof is the same as in
\cite{Cockburn:2010}, but we include a detailed proof here to show
that the proof still holds with sign-changing coefficients.
\begin{theorem}[duality argument]
  \label{thm:duality-argument}
  Suppose that
  $(\qb, u) \in H^s(\Omega; \Bbb{R}^d) \times H^s(\Omega)$, $s \ge 1$,
  and $(\qb_h, u_h, \ubar_h) \in \Vb_h \times W_h \times M_h$ are the
  solutions of \cref{eq:exact} and \cref{eq:hdgmethod},
  respectively. Let $e_u^h$ be as defined in
  \cref{eq:error-decomposition}. If \cref{eq:regularity} holds for the
  solution $(\Phib, \phi)$ of \cref{eq:dualproblem}, then,
  \begin{equation}
    \label{eq:duality-identity}
    \del[0]{e_u^h, \theta }_{\Omega} = \del[0] {\sigma^{-1} (\qb - \qb_h) , \Pi_{\Vb} \Phib - \Phib}_{\Omega} 
    + \del[0] {\Pi_{\Vb} \qb - \qb , \nabla (\phi - \phi_h) }_{\Omega} 
  \end{equation}    
  for any $\phi_h \in W_h$. Additionally, the following
  superconvergence result holds:
  \begin{equation}
    \label{eq:superconvergence}
    \norm{ \Pi_W u - u_h }_{\Omega}
    \le C C_{\rm reg} h \del[1]{ \norm{ \qb - \qb_h }_{\Vb_h} + \norm{ \qb - \Pi_{\Vb} \qb}_{\Omega} } 
  \end{equation}
  if $k \ge 1$.
\end{theorem}
\begin{proof}
  In this proof we abbreviate `integration by parts' by IBP. We will
  furthermore use $\del[0]{\cdot, \cdot}$ and $\lra{\cdot, \cdot}$
  instead of $\del[0]{\cdot, \cdot}_{\Omega}$ and
  $\lra{\cdot, \cdot}_{\partial \mathcal{T}_h}$ for brevity of notation.
  
  We first note that
  \begin{equation}
    \label{eq:normal-continuity}
    \lra{e_{\ubar}^h, \Phib \cdot \nb} = 0.
  \end{equation}
  From \cref{eq:dualproblem} one can obtain
  \begin{align}
    \label{eq:euhthetasuperconv}
    \del[0]{e_u^h, \theta}
    &= \del[0]{e_u^h, \nabla \cdot \Phib} & &
    \\
    \nonumber
    &= - \del[0]{\nabla e_u^h, \Phib} + \lra{e_u^h, \Phib \cdot \nb } & & \text{by IBP}
    \\
    \nonumber
    &= - \del[0]{\nabla e_u^h, \Pi_{\Vb} \Phib} + \lra{e_u^h, \Phib \cdot \nb } & & \text{by \cref{eq:intp1}}
    \\
    \nonumber
    &= \del[0]{e_u^h, \nabla \cdot \Pi_{\Vb} \Phib} + \lra{e_u^h, \del[0]{\Phib - \Pi_{\Vb} \Phib}\cdot \nb} & & \text{by IBP}
    \\
    \nonumber
    &= \del[0]{e_u^h, \nabla \cdot \Pi_{\Vb} \Phib} - \lra{e_u^h, \tau \del[0]{\phi - \Pi_{W} \phi}\cdot \nb} & & \text{by \cref{eq:intp3}}
    \\
    \nonumber
    &= \del[0]{ \sigma^{-1} e_{\qb} , \Pi_{\Vb} \Phib } + \lra{e_{\ubar}^h, \Pi_{\Vb} \Phib \cdot \nb } - \lra{e_u^h, \tau \del[0]{\phi - \Pi_{W} \phi}\cdot \nb} & & \text{by \cref{eq:red-err-eq1}}
    \\
    \nonumber
    &= \del[0]{ \sigma^{-1} e_{\qb} , \Pi_{\Vb} \Phib } + \lra{e_{\ubar}^h, \del[0]{ \Pi_{\Vb} \Phib - \Phib} \cdot \nb } - \lra{e_u^h, \tau \del[0]{\phi - \Pi_{W} \phi}\cdot \nb} & & \text{by \cref{eq:normal-continuity}}
    \\
    \nonumber
    &= \del[0]{ \sigma^{-1} e_{\qb} , \Pi_{\Vb} \Phib } - \lra{e_u^h - e_{\ubar}^h, \tau \del[0]{\phi - \Pi_{W} \phi} }  & & \text{by \cref{eq:intp3}}
    \\
    \nonumber
    &= \del[0]{ \sigma^{-1} e_{\qb} , \Pi_{\Vb} \Phib } - \lra{e_u^h - e_{\ubar}^h, \tau P_M \phi  }
      + \lra{e_u^h - e_{\ubar}^h, \tau \Pi_{W} \phi }
  \end{align}
  where we used the property of the $L^2$ projection $P_M$ in the last
  identity.

  We next note that by \cref{eq:red-err-eq3} and the property of $P_M$,
  \begin{equation*}
   \lra{e_u^h - e_{\ubar}^h, \tau P_M \phi } = - \lra{ e_{\qb}^h \cdot
    \nb , P_M \phi} = - \lra{ e_{\qb}^h \cdot \nb , \phi}. 
  \end{equation*}
  We also have
  \begin{align*}
    \lra{ \tau \del[0]{e_u^h - e_{\ubar}^h}, \Pi_W \phi } 
    &= \del[0]{e_{\qb}^h, \nabla \Pi_W \phi} - \lra{ e_{\qb}^h \cdot \nb, \Pi_W \phi} & &  \text{by \cref{eq:red-err-eq2} }
    \\ 
    &= -\del[0]{\nabla \cdot e_{\qb}^h, \Pi_W \phi} & & \text{by IBP }
    \\
    &= -\del[0]{\nabla \cdot  e_{\qb}^h, \phi }  & & \text{by \cref{eq:intp1}}
    \\
    &= \del[0]{e_{\qb}^h, \nabla \phi } - \lra{ e_{\qb}^h \cdot \nb, \phi}.  & & \text{by IBP}
  \end{align*}
  These two identities combined with \cref{eq:euhthetasuperconv}
  result in
  \begin{equation}
    \label{eq:euhthetasimp}
    \del[0]{e_u^h, \theta}
    = \del[0]{ \sigma^{-1} e_{\qb} , \Pi_{\Vb} \Phib }
    + \del[0]{ e_{\qb}^h , \nabla \phi} .
  \end{equation}
  Note that
  $\del[0]{ \sigma^{-1} e_{\qb}, \Phib } + \del[0]{ e_{\qb}, \nabla
    \phi} = 0$ by \cref{eq:dualproblem_a}. Using this identity and
  \cref{eq:intp1}, we may write \cref{eq:euhthetasimp} as
  \begin{equation*}
    \begin{split}
      \del[0]{e_u^h, \theta}
      &= \del[0]{\sigma^{-1} e_{\qb}, \Pi_{\Vb} \Phib - \Phib} + \del[0]{e_{\qb}^h - e_{\qb}, \nabla \phi }
      \\
      &= \del[0]{\sigma^{-1} \del[0]{\qb - \qb_h}, \Pi_{\Vb} \Phib - \Phib}
      + \del[0]{\Pi_{\Vb} \qb - \qb, \nabla \del[0]{\phi - \phi_h} },      
    \end{split}
  \end{equation*}
  for any $\phi_h \in W_h$. This completes the proof of
  \cref{eq:duality-identity}.
  
  To see \cref{eq:superconvergence}, choose
  $\theta = e_u^h := \Pi_W u - u_h$.  Furthermore, since $k \ge 1$ and
  $\phi_h$ is arbitrary,
  $\norm{ \nabla \del[0]{\phi - \phi_h } }_{\Omega} \le C h
  \norm{\phi}_{2,\Omega}$. Now, applying the Cauchy--Schwarz
  inequality to \cref{eq:duality-identity} gives
  \begin{equation*}
    \begin{split}
      \norm{ \Pi_W u - u_h }_{\Omega}^2
      &\le C h \del[1]{ \norm{ \qb - \qb_h }_{\Vb_h} \norm{\Phib}_{1, \Omega} + \norm{ \qb - \Pi_{\Vb} \qb}_{\Omega} \norm{\phi}_{2,\Omega} }
      \\
      &\le C C_{\rm reg} h \norm{ \Pi_Wu - u_h }_{\Omega} \del[1]{ \norm{ \qb - \qb_h }_{\Vb_h} + \norm{ \qb - \Pi_{\Vb} \qb}_{\Omega} },
    \end{split}
  \end{equation*}
  where we used \cref{eq:regularity} in the last inequality. The
  conclusion follows by dividing both sides by
  $\norm{ \Pi_W u - u_h }_{\Omega}$.
\end{proof}

Following \citep{Stenberg:1991}, we define a post-processed solution
of $u_h$, $u_h^*$, by
\begin{subequations}
  \begin{align}
    \label{eq:post-processing-1}
    \del[0]{ \nabla u_h^*, \nabla v}_{K}
    &= - ( \sigma^{-1} \qb_h , \nabla v )_{K},  & & v \in \mathcal{P}_{k+1} (K),
    \\
    \label{eq:post-processing-2}
    \del[0]{u_h^*, 1}_K &= \del[0]{ u_h, 1}_K,
  \end{align}
  \label{eq:post-processing}
\end{subequations}
for all $K \in \mathcal{T}_h$. The next result shows that the error in
the post-processed solution $u_h^*$ superconverges. Although the proof
is standard, we include it here for completeness.
\begin{theorem}[superconvergence]
  \label{thm:superconvergence}
  Suppose that
  $(\qb, u) \in H^k(\Omega; \Bbb{R}^d) \times H^k(\Omega)$, $k \ge 1$,
  and $(\qb_h, u_h, \ubar_h) \in \Vb_h \times W_h \times M_h$ are the
  solutions of \cref{eq:exact} and \cref{eq:hdgmethod},
  respectively. Let $u_h^*$ be the post-processed solution defined by
  \cref{eq:post-processing}. The following result holds:
  \begin{equation*}
    \norm{u - u_h^*}_{\Omega}
    \le C h \del[1]{\norm{ \qb - \qb_h}_{\Vb_h} + \norm{ \qb - \Pi_{\Vb}\qb}_{\Omega}}  + C h^{k+2} \norm{u }_{k+2, \Omega} .  
  \end{equation*}
\end{theorem}
\begin{proof}
  Let
  $\mathcal{P}_m(\mathcal{T}_h) := \cbr[0]{v \in L^2(\Omega) \,:\,
    v|_K \in \mathcal{P}_m(K) \ \forall K \in \mathcal{T}_h}$. Let
  $\Pi: H^1(\Omega) \rightarrow \mathcal{P}_{k+1} (\mathcal{T}_h)$ be
  an interpolation such that $P_0 \Pi = P_0$ and
  $\norm{ v - \Pi v }_{\Omega} + h \norm{ \nabla v - \nabla \Pi
    v}_{\Omega} \le C h^{k+2} \norm{ v }_{k+2, \Omega}$, where $P_0$
  is the $L^2$ projection to $\mathcal{P}_0(\mathcal{T}_h)$.  By the
  triangle inequality, it suffices to estimate
  $\norm{\Pi u - u_h^*}_{\Omega}$. Again by the triangle inequality,
  it is sufficient to estimate $\norm{P_0 (\Pi u - u_h^*)}_{\Omega}$
  and $\norm{ (I - P_0) (\Pi u - u_h^*)}_{\Omega}$.

  Since $P_0 (\Pi u - u_h^*) = P_0 u - P_0 u_h = P_0 (\Pi_W u - u_h)$
  by \cref{eq:post-processing-2} and the definition of $\Pi_W$, we
  find
  $\norm{P_0 (\Pi u - u_h^*) }_{\Omega} \le C h \del[1]{ \norm{ \qb -
      \qb_h}_{\Vb_h} + \norm{ \qb - \Pi_{\Vb}\qb}_{\Omega} }$ by
  \cref{thm:duality-argument}.

  To estimate $\norm{ (I - P_0) (\Pi u - u_h^*)}_{\Omega}$, recall
  that $\qb = - \sigma \nabla u$. By \cref{eq:post-processing-1} we
  have
  \begin{equation*}
    \del[0]{ \nabla( u - u_h^*), \nabla v}_{\Omega} = - \del[0]{ \sigma^{-1} (\qb - \qb_h), \nabla v}_{\Omega} ,
    \qquad v \in \mathcal{P}_{k+1}(\mathcal{T}_h).
  \end{equation*}
  Taking $v = (I - P_0) (\Pi u - u_h^*)$, we obtain
  \begin{equation*}
    \begin{split}
      \norm{\nabla (I - P_0)  (\Pi u - u_h^*)}_{\Omega}^2
      =& - \del[0]{ \nabla (u - \Pi u), \nabla (I - P_0)  (\Pi u - u_h^*)}_{\Omega}
      \\
      &- \del[0]{ \sigma^{-1} (\qb - \qb_h), \nabla (I - P_0)  (\Pi u - u_h^*)}_{\Omega} .      
    \end{split}
  \end{equation*}
  The Cauchy--Schwarz inequality and the assumption on $\Pi$ gives
  \begin{align*}
      \norm{\nabla (I - P_0)  (\Pi u - u_h^*)}_{\Omega}
      \le& \norm{ \nabla (u - \Pi u)}_{\Omega} +  \norm{ \sigma^{-1} (\qb - \qb_h)}_{\Omega}  \\
      \le& h^{k+1} \norm{u}_{k+2, \Omega} + C \norm{ \qb - \qb_h }_{\Vb_h} .  
  \end{align*}
  The desired estimate follows by combining this estimate and the
  estimate
  \begin{equation*}
    \norm{(I - P_0)  (\Pi u - u_h^*)}_{\Omega} \le Ch \norm{\nabla (I - P_0)  (\Pi u - u_h^*)}_{\Omega}
  \end{equation*}
  obtained by an element-wise Poincare inequality. 
\end{proof}

\section{Numerical examples}
\label{sec:num_examples}
In this section we verify the analysis of \cref{sec:error_analysis}
and \cref{sec:superconvergence} via numerical experiments. All the
examples have been implemented using the NGSolve finite element
library \cite{Schoberl:2014}.

\subsection{The cavity problem}
\label{ss:cavity}
We consider the symmetric cavity problem of \cite[Section
6]{Chesnel:2013}.

Let $\Omega = (-1,1) \times (0,1)$, $\Omega_+ := (-1,0) \times (0,1)$,
and $\Omega_- := (0,1) \times (0,1)$. We choose the source term
$f \in L^2(\Omega)$ such that the exact solution to
\cref{eq:variationalForm} is given by
\begin{equation}
  \label{eq:exactsol}
  u(x_1, x_2) :=
  \begin{cases}
    ((x_1+1)^2 - (\sigma_++\sigma_-)^{-1}(2\sigma_+ + \sigma_-)(x_1+1))\sin(\pi x_2) & \text{on}\ \Omega_+,
    \\
    (\sigma_+ + \sigma_-)^{-1}\sigma_+(x_1 - 1)\sin(\pi x_2) & \text{on}\ \Omega_-.
  \end{cases}
\end{equation}
We set $\sigma_+ = 1$ and $\kappa_{\sigma} = -1.001$.

We consider both a symmetric and a non-symmetric mesh, see
\cref{fig:meshes}.
\begin{figure}
  \centering
  \subfloat[A symmetric mesh. \label{fig:mesh_sym}]{\includegraphics[width=0.48\textwidth]{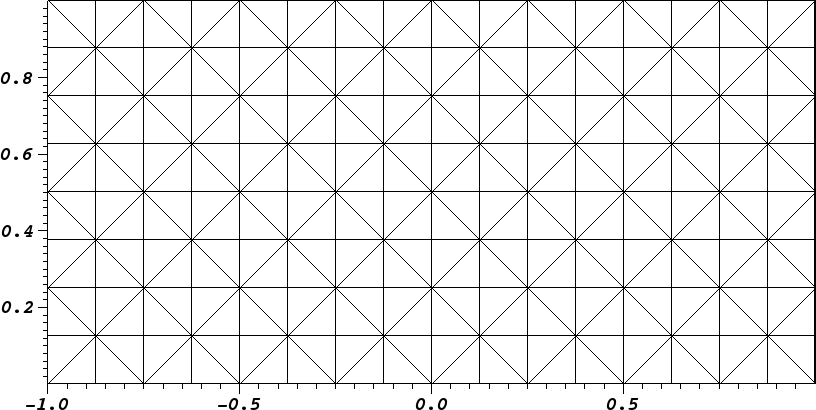}}
  \quad
  \subfloat[A non-symmetric mesh. \label{fig:mesh_nonsym}]{\includegraphics[width=0.48\textwidth]{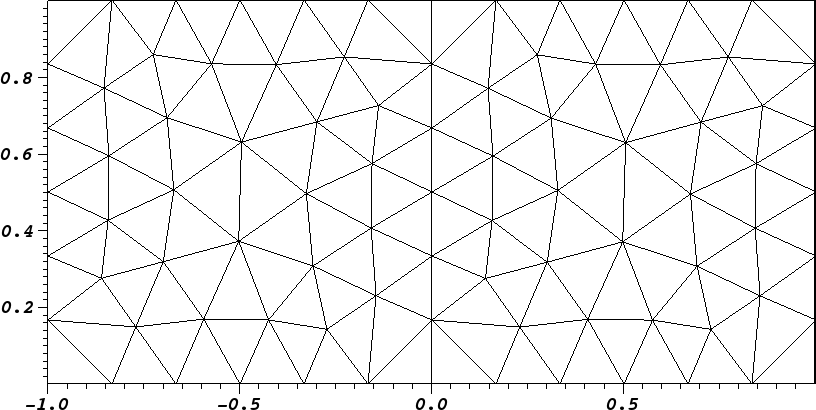}}
  \caption{A symmetric mesh and non-symmetric mesh used in the numerical examples described in~\cref{sec:num_examples}.}
  \label{fig:meshes}
\end{figure}
We will compare the results obtained by the HDG method
\cref{eq:hdgmethod} to those obtained by an $H^1$-conforming finite
element method (CG) as analyzed in \cite{Chesnel:2013}. For the HDG
method the penalty parameter is set to $\tau = 1$ on $\Omega_+$ and
$\tau = -1$ on $\Omega_-$. We take $k=0,1,2,3$ as polynomial degree.

On the symmetric mesh we observe from \cref{tab:symmetricmesh-hdg}
that the primal variable computed using the $H^1$-conforming finite
element method converges optimally in the $L^2$-norm for $k \ge
1$. Optimal rates of convergence in the $L^2$-norm are obtained with
the HDG method for the primal variable as well as the vector variable
for $k \ge 0$. Furthermore, we observe superconvergence in the
post-processed primal variable for $k \ge 1$ when the HDG method is
used. These HDG results verify our analysis in
\cref{sec:error_analysis} and \cref{sec:superconvergence}.

On the non-symmetric mesh we observe from
\cref{tab:nonsymmetricmesh-hdg} irregular rates of convergence in the
$L^2$-norm in the primal variable when using the $H^1$-conforming
finite element method for $k=1$ and $k=2$. Similar behaviour is
observed also in \cite{Chesnel:2013}. Near optimal rates of
convergence in the $L^2$-norm are obtained with this method for
$k=3$. Using the HDG method we observe optimal rates of convergence in
the $L^2$-norm for the errors in $u_h$ and $\boldsymbol{q}_h$ and
superconvergence for the error in $u_h^*$ for $k=2$ and $k=3$. The
rates of convergence in the $L^2$-norm using the HDG method when $k=0$
or $k=1$ show the same irregular behaviour as when using the
$H^1$-conforming finite element method. We remark, though, that the
HDG method always shows a smaller error on the same mesh compared to
the $H^1$-conforming finite element method.

In \cref{fig:nonsymm_contour} we compare the solution $u_h$ computed
using the $H^1$-conforming finite element method and HDG method with
the exact solution for different values of $\kappa_{\sigma}$. We use
$k=2$ and compute the solution on a non-symmetric mesh consisting of
8960 simplicial elements. We observe that the closer $\kappa_{\sigma}$
is to $-1$ (for which the problem is not well-posed), the worse the
$H^1$-conforming finite element compares to the exact solution. The
HDG method compares well with the exact solution for all tested values
of $\kappa_{\sigma}$. We observe the same behaviour in
\cref{fig:nonsymm_slice} in which we plot the solution along the line
$x_2=0.5$. We indeed saw in \cref{tab:nonsymmetricmesh-hdg} that the
HDG method outperforms the $H^1$-conforming finite element method on
the non-symmetric mesh.

\begin{table}
  {\small
  \centering {
    \begin{tabular}{c|cccccccc}
      \hline
      Cells & CG $e_u$ & r & HDG $e_u$ & r & HDG $e_{\boldsymbol{q}}$ & r & HDG $e_{u^*}$ & r \\
      \hline
      \hline
            & \multicolumn{8}{c}{$k=0$} \\
      \hline
   256 & - & -  & 9.9e+0 & -   & 1.9e+2 & -   & 9.9e+0 & -    \\ 
  1024 & - & -  & 5.0e+0 & 1.0 & 9.9e+1 & 1.0 & 5.0e+0 & 1.0 \\ 
  4096 & - & -  & 2.5e+0 & 1.0 & 5.0e+1 & 1.0 & 2.5e+0 & 1.0 \\ 
 16384 & - & -  & 1.3e+0 & 1.0 & 2.5e+1 & 1.0 & 1.3e+0 & 1.0 \\ 
 65536 & - & -  & 6.4e-1 & 1.0 & 1.3e+1 & 1.0 & 6.4e-1 & 1.0 \\ 
262144 & - & -  & 3.2e-1 & 1.0 & 6.3e+0 & 1.0 & 3.2e-1 & 1.0 \\ 
      \hline
            & \multicolumn{8}{c}{$k=1$} \\
      \hline
   256  & 1.0e+1 & -    & 6.9e+0 & -   & 1.5e+1 & -   & 2.6e-1 & -    \\ 
  1024  & 2.5e+0 & 2.0  & 1.7e+0 & 2.0 & 3.9e+0 & 2.0 & 3.3e-2 & 3.0 \\ 
  4096  & 6.3e-1 & 2.0  & 4.1e-1 & 2.0 & 9.7e-1 & 2.0 & 4.2e-3 & 3.0 \\ 
 16384  & 1.6e-1 & 2.0  & 1.0e-1 & 2.0 & 2.4e-1 & 2.0 & 5.3e-4 & 3.0 \\ 
 65536  & 3.9e-2 & 2.0  & 2.5e-2 & 2.0 & 6.1e-2 & 2.0 & 6.6e-5 & 3.0 \\ 
262144  & 9.8e-3 & 2.0  & 6.4e-3 & 2.0 & 1.5e-2 & 2.0 & 8.3e-6 & 3.0 \\ 
      \hline
            & \multicolumn{8}{c}{$k=2$} \\
      \hline
   256 & 2.9e-1 & -   & 2.5e-1 & -   & 5.3e-1 & -   & 7.3e-3 & -    \\ 
  1024 & 3.7e-2 & 3.0 & 3.1e-2 & 3.0 & 6.7e-2 & 3.0 & 4.6e-4 & 4.0 \\ 
  4096 & 4.6e-3 & 3.0 & 3.8e-3 & 3.0 & 8.5e-3 & 3.0 & 2.9e-5 & 4.0 \\ 
 16384 & 5.8e-4 & 3.0 & 4.7e-4 & 3.0 & 1.1e-3 & 3.0 & 1.8e-6 & 4.0 \\ 
 65536 & 7.2e-5 & 3.0 & 5.9e-5 & 3.0 & 1.3e-4 & 3.0 & 1.1e-7 & 4.0 \\ 
262144 & 9.0e-6 & 3.0 & 7.3e-6 & 3.0 & 1.7e-5 & 3.0 & 7.1e-9 & 4.0 \\
      \hline
            & \multicolumn{8}{c}{$k=3$} \\
      \hline
   256 & 7.3e-3 & -   & 7.1e-3 & -   & 1.5e-2 & -   & 1.6e-4  & -    \\ 
  1024 & 4.6e-4 & 4.0 & 4.4e-4 & 4.0 & 9.3e-4 & 4.0 & 5.1e-6  & 5.0 \\ 
  4096 & 2.9e-5 & 4.0 & 2.7e-5 & 4.0 & 5.8e-5 & 4.0 & 1.6e-7  & 5.0 \\ 
 16384 & 1.8e-6 & 4.0 & 1.7e-6 & 4.0 & 3.6e-6 & 4.0 & 5.0e-9  & 5.0 \\ 
 65536 & 2.2e-7 & 3.0 & 1.0e-7 & 4.0 & 2.3e-7 & 4.0 & 1.9e-10 & 4.7 \\ 
      \hline            
    \end{tabular}
  } }
\caption{Rates of convergence in the $L^2$-norm of CG and HDG on a
  symmetric mesh for different polynomial degree $k$, see
  \cref{ss:cavity}.}
  \label{tab:symmetricmesh-hdg}
\end{table}

\begin{table}
  {\small
  \centering {
    \begin{tabular}{c|cccccccc}
      \hline
      Cells & CG $e_u$ & r & HDG $e_u$ & r & HDG $e_{\boldsymbol{q}}$ & r & HDG $e_{u^*}$ & r \\
      \hline
      \hline
            & \multicolumn{8}{c}{$k=0$} \\
      \hline
   140 & - & -  & 3.0e+2 & -   & 1.5e+3 & -   & 3.0e+2 & -    \\ 
   560 & - & -  & 2.7e+2 & 0.2 & 3.3e+3 & -1.1 & 2.7e+2 & 0.2 \\ 
  2240 & - & -  & 3.0e+2 & -0.2 & 2.4e+3 & 0.4 & 3.0e+2 & -0.2 \\ 
  8960 & - & -  & 2.8e+2 & 0.1 & 2.2e+3 & 0.1 & 2.8e+2 & 0.1 \\ 
 35840 & - & -  & 4.4e+3 & -4.0 & 5.2e+4 & -4.6 & 4.4e+3 & -4.0 \\ 
143360 & - & -  & 1.8e+2 & 4.6 & 5.8e+3 & 3.2 & 1.8e+2 & 4.6 \\ 
573440 & - & -  & 1.6e+1 & 3.5 & 5.6e+2 & 3.4 & 1.6e+1 & 3.5 \\
      \hline
            & \multicolumn{8}{c}{$k=1$} \\
      \hline
   140  & 2.7e+3 & -     & 2.5e+1 & -   & 2.4e+2 & -   & 2.2e+1 & -    \\ 
   560  & 1.3e+2 & 4.3   & 7.6e+1 & -1.6 & 1.1e+3 & -2.3 & 7.6e+1 & -1.8 \\ 
  2240  & 2.5e+1 & 2.4   & 1.5e+0 & 5.7 & 3.6e+1 & 5.0 & 1.2e+0 & 6.0 \\ 
  8960  & 1.2e+1 & 1.1   & 8.7e-1 & 0.8 & 3.3e+1 & 0.1 & 8.5e-1 & 0.5 \\ 
 35840  & 6.1e+0 & 1.0   & 1.2e+0 & -0.4 & 7.9e+1 & -1.3 & 1.2e+0 & -0.5 \\ 
143360  & 2.1e+0 & 1.5   & 1.3e-2 & 6.5 & 4.2e-1 & 7.5 & 2.4e-3 & 8.9 \\ 
      \hline
            & \multicolumn{8}{c}{$k=2$} \\
      \hline
   140 & 2.3e+2 & -     & 1.1e+0 & -   & 2.0e+1 & -   & 9.1e-1 & -    \\ 
   560 & 1.0e+3 & -2.1  & 1.0e-1 & 3.3 & 3.1e+0 & 2.7 & 5.5e-2 & 4.0 \\ 
  2240 & 7.8e-1 & 10.3  & 1.1e-2 & 3.2 & 2.6e-1 & 3.6 & 2.7e-3 & 4.3 \\ 
  8960 & 7.9e-2 & 3.3   & 1.4e-3 & 3.0 & 1.8e-2 & 3.8 & 9.5e-5 & 4.9 \\ 
 35840 & 5.2e-2 & 0.6   & 1.7e-4 & 3.0 & 1.6e-3 & 3.5 & 4.1e-6 & 4.5 \\ 
143360 & 7.8e-4 & 6.1   & 2.2e-5 & 3.0 & 2.6e-4 & 2.6 & 3.8e-7 & 3.4 \\
      \hline
            & \multicolumn{8}{c}{$k=3$} \\
      \hline      
   140 & 7.1e-1 & -   & 3.0e-2 & -   & 8.6e-1 & -   & 2.0e-2  & -    \\ 
   560 & 7.0e-2 & 3.3 & 2.5e-3 & 3.6 & 8.5e-2 & 3.4 & 1.5e-3  & 3.8 \\ 
  2240 & 2.6e-3 & 4.7 & 1.3e-4 & 4.3 & 3.9e-3 & 4.5 & 3.5e-5  & 5.4 \\ 
  8960 & 7.9e-5 & 5.1 & 7.7e-6 & 4.1 & 1.7e-4 & 4.5 & 7.6e-7  & 5.5 \\ 
 35840 & 3.3e-6 & 4.6 & 4.8e-7 & 4.0 & 7.3e-6 & 4.5 & 1.6e-8  & 5.5 \\ 
143360 & 2.7e-7 & 3.6 & 3.0e-8 & 4.0 & 3.2e-7 & 4.5 & 7.1e-10 & 4.5 \\ 
      \hline            
    \end{tabular}
  } }
\caption{Rates of convergence in the $L^2$-norm of CG and HDG on a
  non-symmetric mesh for different polynomial degree $k$, see
  \cref{ss:cavity}.}
  \label{tab:nonsymmetricmesh-hdg}
\end{table}

\begin{figure}[tbp]
  \begin{center}
    \subfloat[$H^1$-FEM, $\kappa_{\sigma}=-1.001$.]{\includegraphics[width=.45\linewidth]{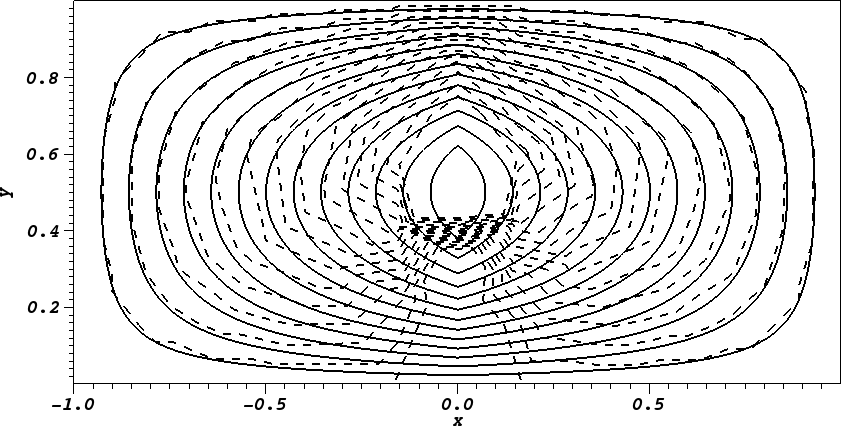}}
    \quad
    \subfloat[HDG, $\kappa_{\sigma}=-1.001$.]{\includegraphics[width=.45\linewidth]{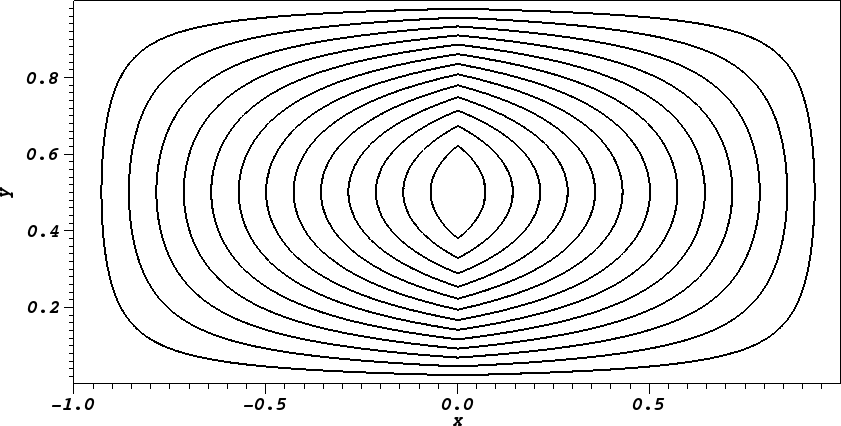}}
    \\
    \subfloat[$H^1$-FEM, $\kappa_{\sigma}=-1.1$.]{\includegraphics[width=.45\linewidth]{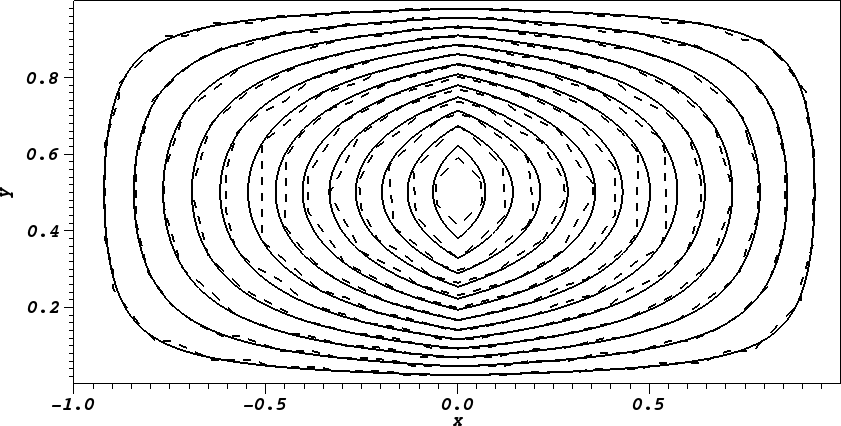}}
    \quad
    \subfloat[HDG, $\kappa_{\sigma}=-1.1$.]{\includegraphics[width=.45\linewidth]{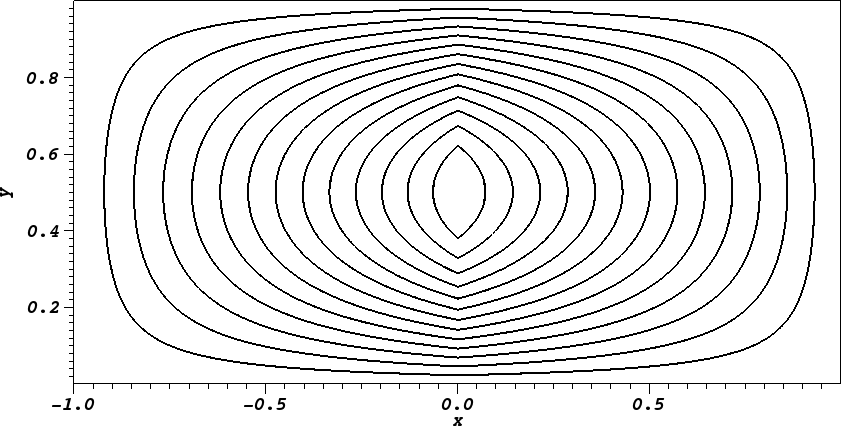}}
    \\
    \subfloat[$H^1$-FEM, $\kappa_{\sigma}=-2$.]{\includegraphics[width=.45\linewidth]{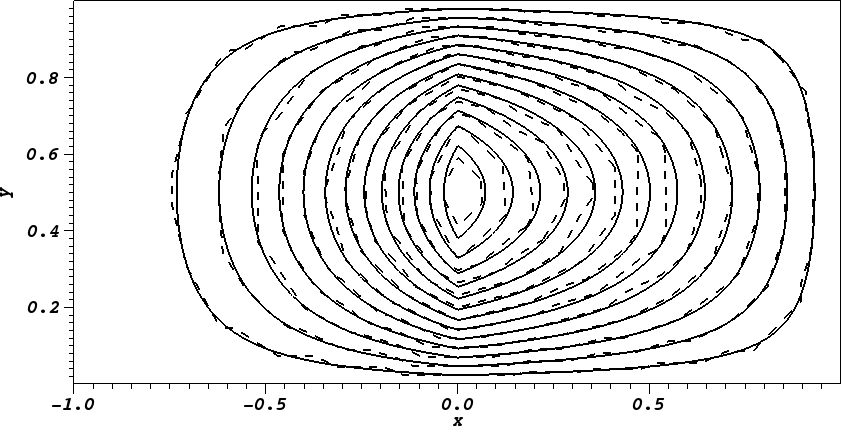}}
    \quad
    \subfloat[HDG, $\kappa_{\sigma}=-2$.]{\includegraphics[width=.45\linewidth]{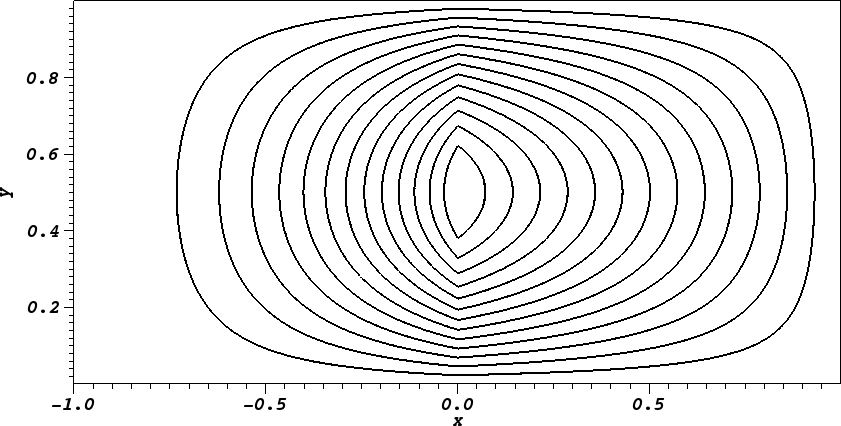}}
    \caption{A contour plot of $u_h$ comparing the $H^1$-conforming
      finite element method and HDG method (dashed lines) with the
      exact solution (solid lines) for different values of
      $\kappa_{\sigma}$, see \cref{ss:cavity}.}
    \label{fig:nonsymm_contour}
  \end{center}
\end{figure}

\begin{figure}[tbp]
  \begin{center}
    \subfloat[$H^1$-FEM, $\kappa_{\sigma}=-1.001$.]{\includegraphics[width=.45\linewidth]{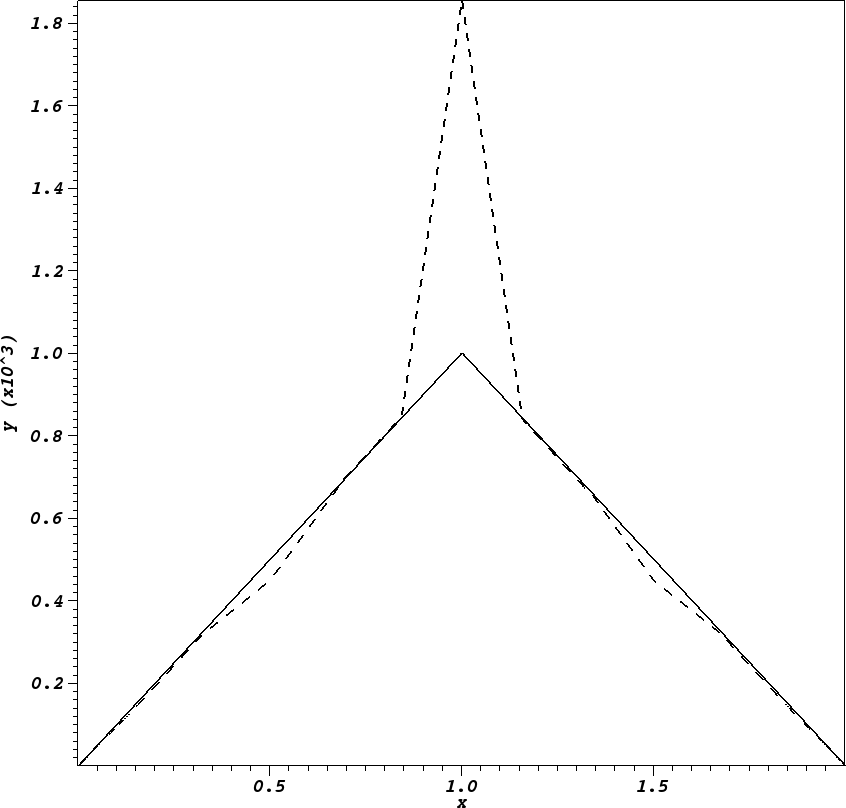}}
    \quad
    \subfloat[HDG, $\kappa_{\sigma}=-1.001$.]{\includegraphics[width=.45\linewidth]{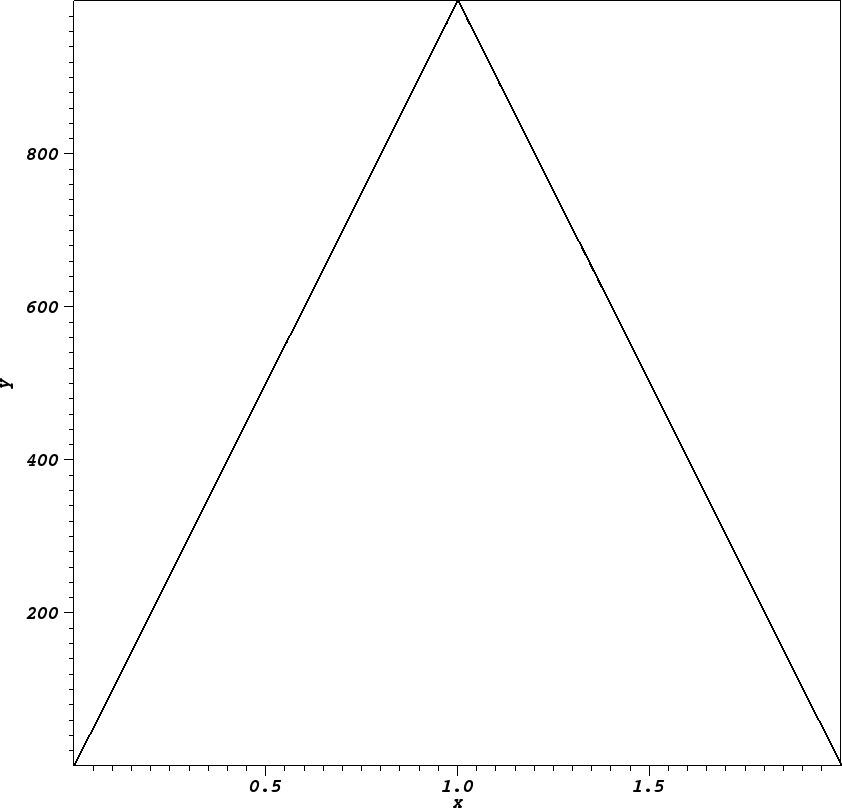}}
    \\
    \subfloat[$H^1$-FEM, $\kappa_{\sigma}=-1.1$.]{\includegraphics[width=.45\linewidth]{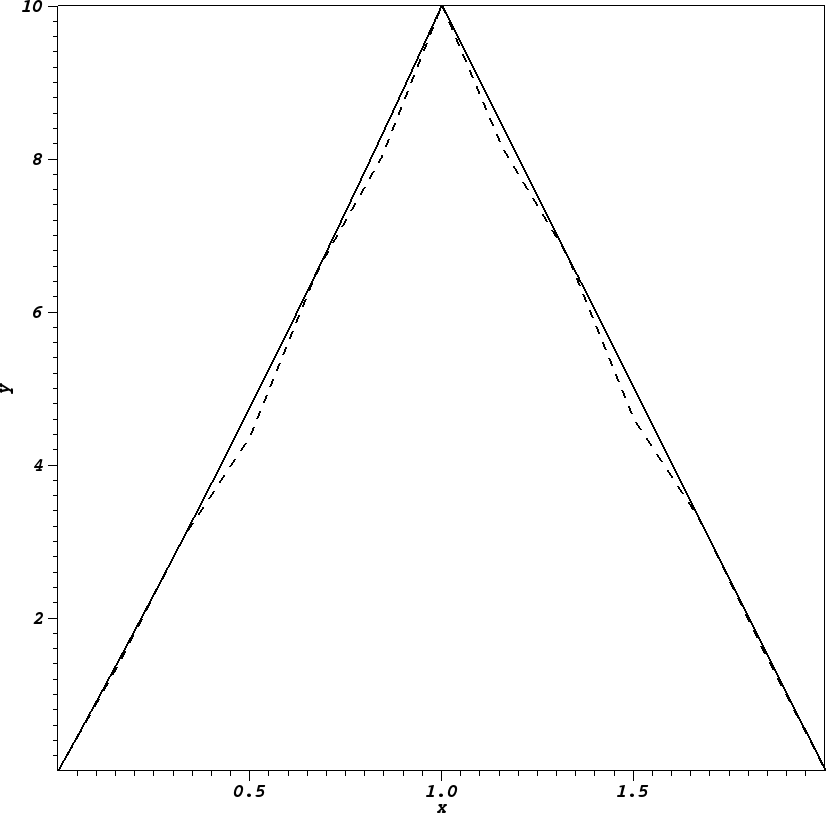}}
    \quad
    \subfloat[HDG, $\kappa_{\sigma}=-1.1$.]{\includegraphics[width=.45\linewidth]{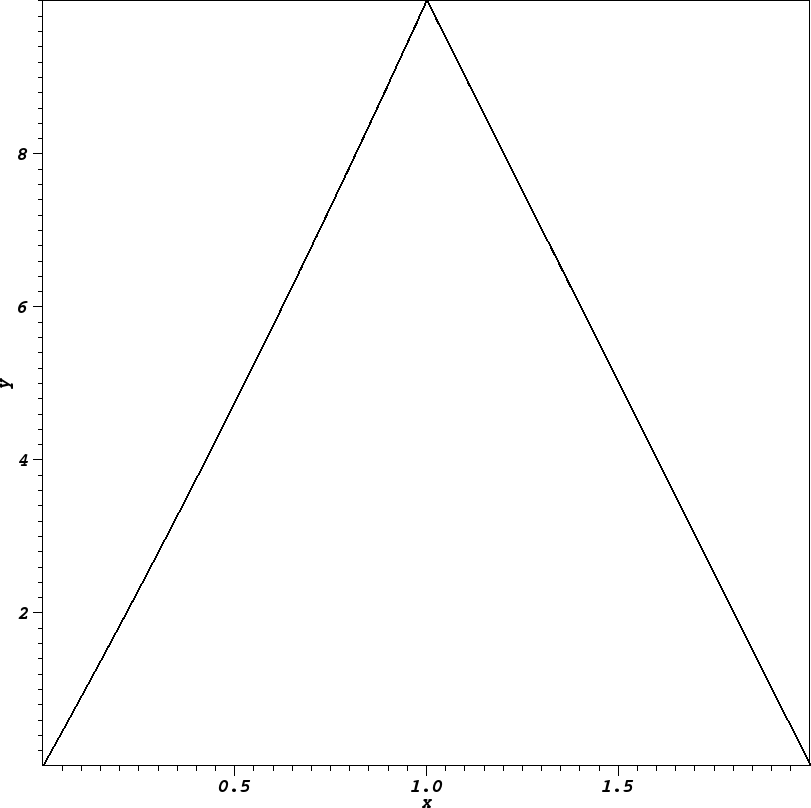}}
    \\
    \subfloat[$H^1$-FEM, $\kappa_{\sigma}=-2$.]{\includegraphics[width=.45\linewidth]{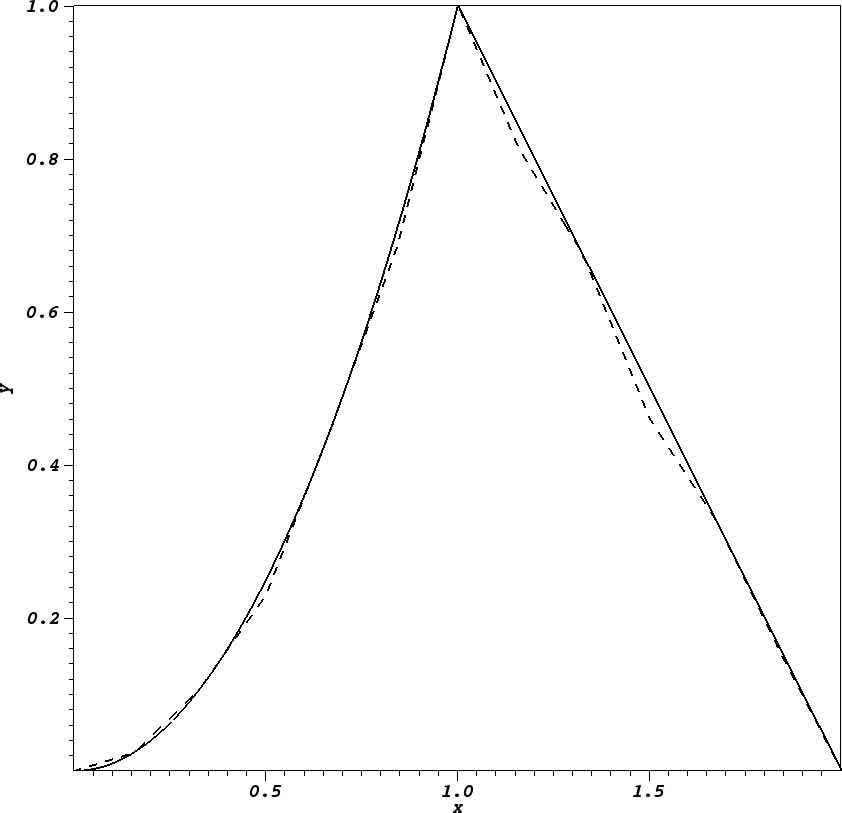}}
    \quad
    \subfloat[HDG, $\kappa_{\sigma}=-2$.]{\includegraphics[width=.45\linewidth]{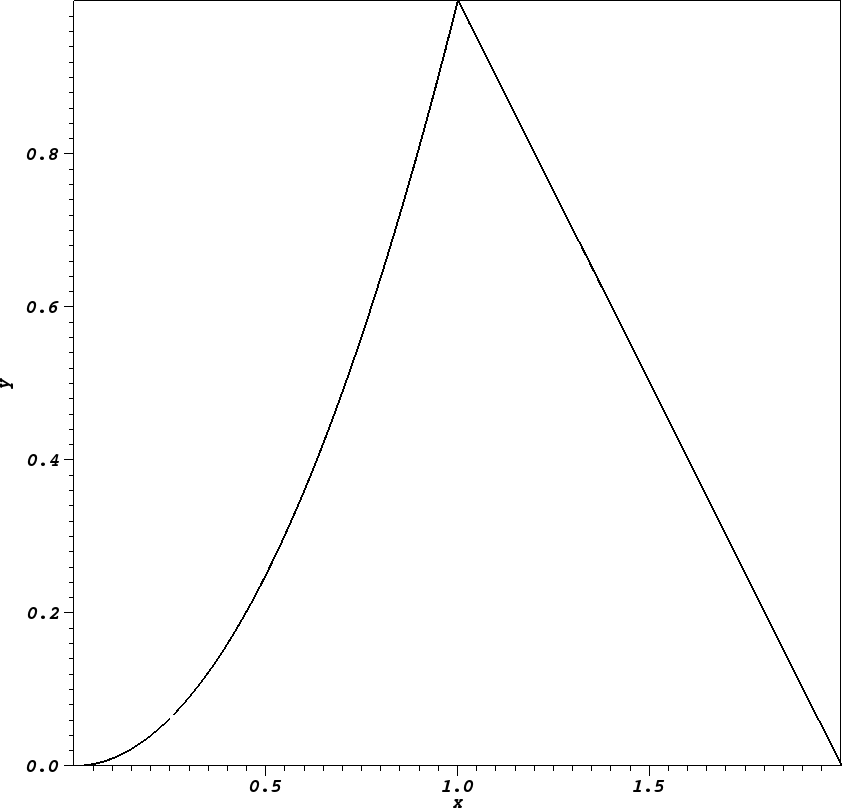}}
    \caption{Comparing the solution exact solution $u$ (solid lines)
      along the line $x_2=0.5$ with that computed using the
      $H^1$-conforming finite element method and the HDG method
      (dashed lines) for different values of $\kappa_{\sigma}$, see
      \cref{ss:cavity}.}
    \label{fig:nonsymm_slice}
  \end{center}
\end{figure}

\subsection{Meta-material layer}
\label{ss:metamat}
We next consider a problem motivated by that given in \cite[Section
6.1]{Chung:2013}. For this, let $\Omega = (0, 5) \times (0, 2)$. The
definition of $\Omega_+$ and $\Omega_-$ are shown in
\cref{fig:domain_plus_min_def}. For this setup it is known that
\cref{eq:problem} is well-posed if
$\kappa_{\sigma} \in (-\infty, 0) \backslash\sbr[0]{\kappa_{\min},
  \kappa_{\max}}$ where $\kappa_{\min} \approx -1.46$ and
$\kappa_{\max} \approx -0.69$, \cite[Section 3.2]{Chesnel:2013}.

\begin{figure}
\begin{center}
  \begin{tikzpicture}[
    important line/.style={thick},
    big dot/.style={circle, draw, inner sep=0pt, minimum size=3mm, fill=yellow},
    point/.style = {draw, circle,  ball color = blue, inner sep = 2pt, outer sep = 0pt, minimum size=2mm},
    pile/.style={thick, ->, >=stealth', shorten <=2pt, shorten  >=2pt},
    scale=1.5
    ]

    \draw[important line] (0,0)--(5,0);
    \draw[important line] (5,0)--(5,2);
    \draw[important line] (5,2)--(0,2);
    \draw[important line] (0,2)--(0,0);
    \draw[important line] (1,0)--(1.3,1);
    \draw[important line] (1.3,1)--(1,2);
    \draw[important line] (3,0)--(3.3,1);
    \draw[important line] (3.3,1)--(3,2);

    \draw[dashed] (1.3,0)--(1.3,0.1);
    \draw[dashed] (3.3,0)--(3.3,0.1);
    \draw[dashed] (0,1)--(0.1,1);
    
    \draw[->] (5,0)--(5.5,0);
    \draw[->] (0,2)--(0,2.5);

    \node[] at (0,2.6){$x_2$};
    \node[] at (5.7,0){$x_1$};

    \node[] at (0.7,1){$\Omega_+$};
    \node[] at (2.2,1){$\Omega_-$};
    \node[] at (4.2,1){$\Omega_+$};
    
    \node[] at (0, -0.2){0};
    \node[] at (1, -0.2){1};
    \node[] at (1.3, -0.2){1.3};
    \node[] at (3, -0.2){3};
    \node[] at (3.3, -0.2){3.3};
    \node[] at (5, -0.2){5};
    \node[] at (-0.2, 1){1};
    \node[] at (-0.2, 2){2};

  \end{tikzpicture}
\end{center}
\caption{The domain $\Omega = (0, 5) \times (0, 2)$ is subdivided in
  $\Omega_+$ and $\Omega_-$, see \cref{ss:metamat}.}
\label{fig:domain_plus_min_def}
\end{figure}
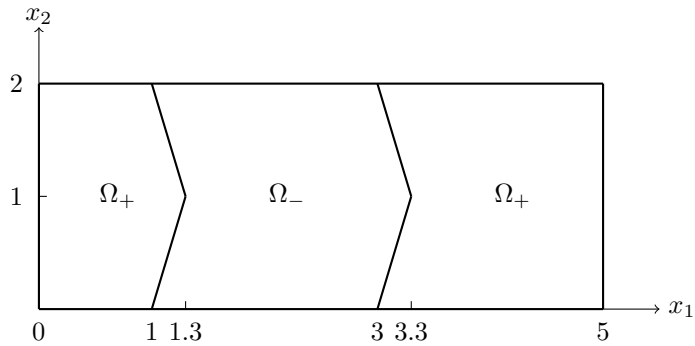

We set $\sigma_+ = 1$ and 
\begin{equation}
  f(x_1, x_2) :=
  \begin{cases}
    \sin(\pi x_2/2) & \text{in } \Omega_+ \text{ if } x_1 < 1.3, \\
    0 & \text{otherwise}.
  \end{cases}
\end{equation}
We consider three different values for $\kappa_{\sigma}$:
$\kappa_{\sigma} = -1.5$, $\kappa_{\sigma} = -1.6$, and
$\kappa_{\sigma} = -2$. 

We compute the solution using both the HDG method \cref{eq:hdgmethod}
and an $H^1$-conforming finite element method. For this we take $k=3$
and consider the solution on a mesh consisting of 24576 simplicial
elements. We plot and compare contour plots of $u_h$ obtained by both
methods in \cref{fig:contour_slice}. We observe that the closer
$\kappa_{\sigma}$ is to the interval
$\sbr[0]{\kappa_{\min}, \kappa_{\max}}$, the larger the difference in
solution between the HDG and $H^1$-conforming finite element
method. This is observed also when plotting the solution on the line
$x_2=1$. When $\kappa_{\sigma}=-1.5$ we observe a larger undershoot at
$x_1=3.3$ with the $H^1$-conforming finite element method than when
using the HDG method. When $\kappa_{\sigma} = -1.6$ and
$\kappa_{\sigma} = -2$, i.e., $\kappa_{\sigma}$ is further away from
the interval $\sbr[0]{\kappa_{\min}, \kappa_{\max}}$, the difference
in solution between both discretizations is smaller. These results are
in agreement with those observed in \cref{ss:cavity}.

\begin{figure}[tbp]
  \begin{center}
    \subfloat[Contour plot of $u_h$ for $\kappa_{\sigma}=-1.5$.]{\includegraphics[width=.65\linewidth]{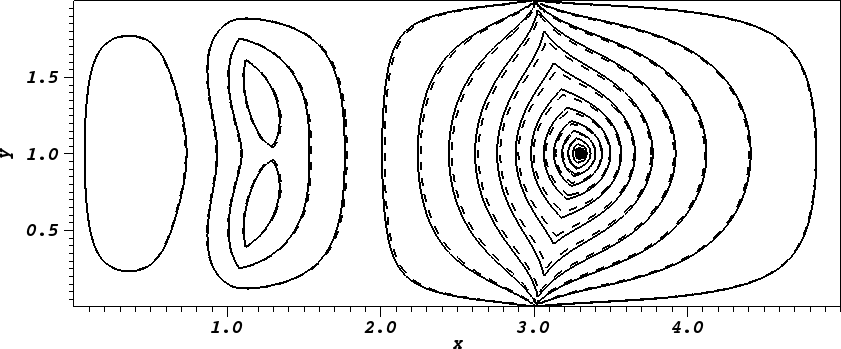}}
    \quad
    \subfloat[Slice through $x_2=1$ for $\kappa_{\sigma}=-1.5$.]{\includegraphics[width=.3\linewidth]{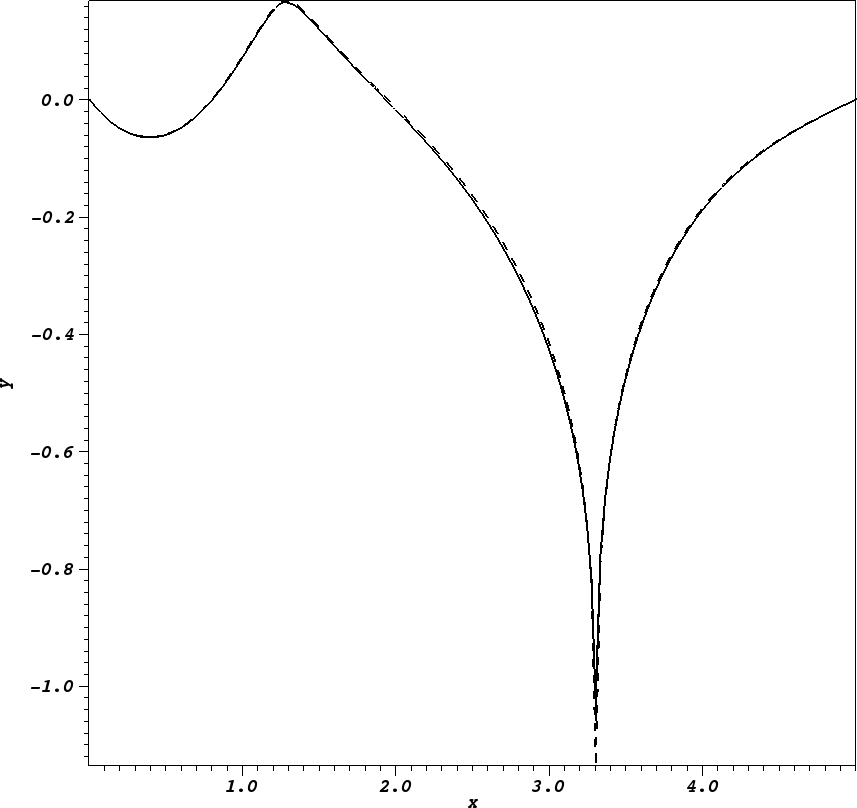}}
    \\
    \subfloat[Contour plot of $u_h$ for $\kappa_{\sigma}=-1.6$.]{\includegraphics[width=.65\linewidth]{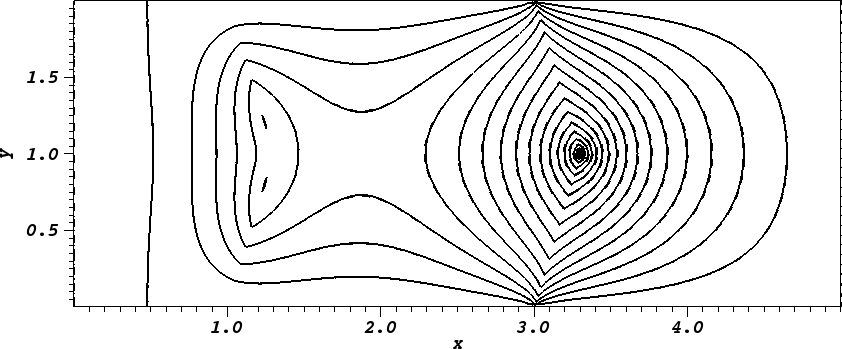}}
    \quad
    \subfloat[Slice through $x_2=1$ for $\kappa_{\sigma}=-1.6$.]{\includegraphics[width=.3\linewidth]{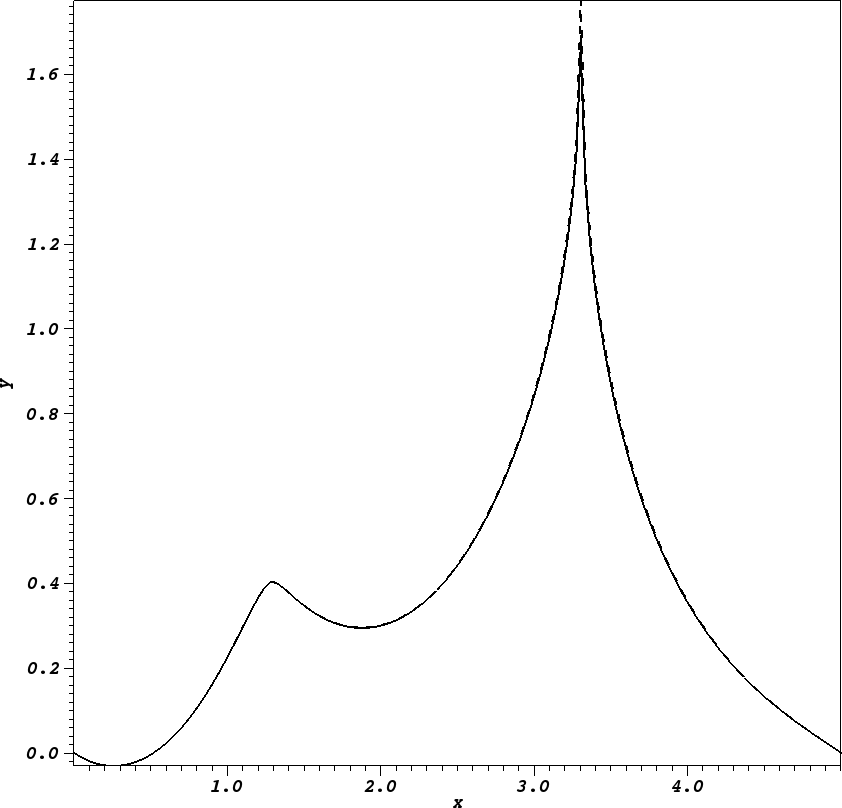}}
    \\
    \subfloat[Contour plot of $u_h$ for $\kappa_{\sigma}=-2$.]{\includegraphics[width=.65\linewidth]{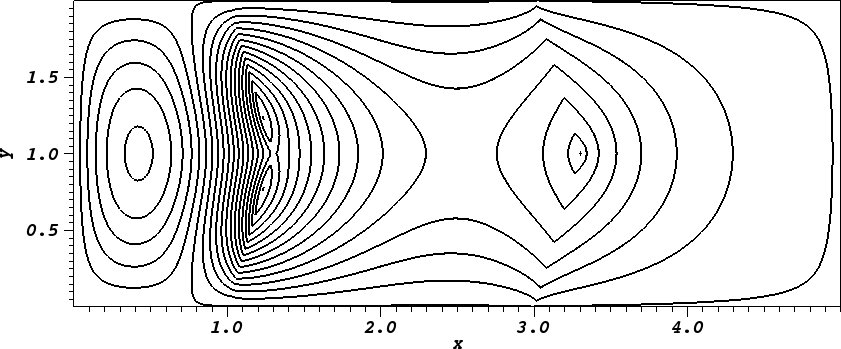}}
    \quad
    \subfloat[Slice through $x_2=1$ for $\kappa_{\sigma}=-2$.]{\includegraphics[width=.3\linewidth]{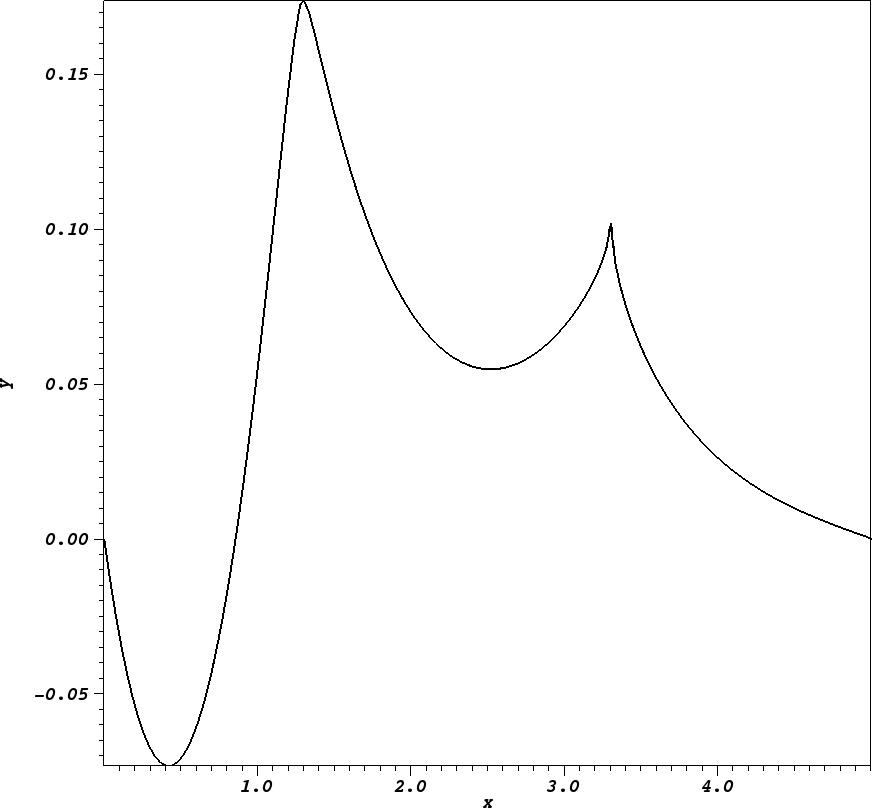}}
    \caption{A contour plot of $u_h$ a slice through $x_2=1$ comparing
      the solution obtained by the HDG method (solid lines) to an
      $H^1$-conforming finite element method (dashed lines) for
      different values of $\kappa_{\sigma}$, see \cref{ss:metamat}.}
    \label{fig:contour_slice}
  \end{center}
\end{figure}

\section{Conclusions}
\label{sec:conclusions}

In this paper we introduced an HDG method for the solution of
Poisson-type problems with a sign-changing coefficient. Well-posedness
of our method does not rely on conditions on domain geometry and the
ratio of the negative and positive coefficients. We presented a novel
error analysis of the HDG method that avoids the Aubin--Nitsche
duality argument. Under suitable regularity assumptions, a
superconvergence result is obtained by a duality argument. Numerical
experiments show that the HDG method outperforms an $H^1$-conforming
finite element method for this problem.

\bibliographystyle{plainnat}
\bibliography{references}

\begin{thebibliography}{20}
\providecommand{\natexlab}[1]{#1}
\providecommand{\url}[1]{\texttt{#1}}
\expandafter\ifx\csname urlstyle\endcsname\relax
  \providecommand{\doi}[1]{doi: #1}\else
  \providecommand{\doi}{doi: \begingroup \urlstyle{rm}\Url}\fi

\bibitem[Arnold and Brezzi(1985)]{Arnold:1985}
D.~N. Arnold and F.~Brezzi.
\newblock Mixed and nonconforming finite element methods: implementation,
  postprocessing and error estimates.
\newblock \emph{RAIRO Mod\'el. Math. Anal. Num\'er.}, 19\penalty0 (1):\penalty0
  7--32, 1985.
\newblock \doi{10.1051/m2an/1985190100071}.

\bibitem[{Bonnet-Ben Dhia} et~al.(2010){Bonnet-Ben Dhia}, Jr., and
  Zw\"{o}lf]{Zwolf:2010}
A.-S. {Bonnet-Ben Dhia}, P.~Ciarlet Jr., and C.~M. Zw\"{o}lf.
\newblock Time harmonic wave diffraction problems in materials with
  sign-shifting coefficients.
\newblock \emph{J. Comput. Appl. Math.}, 234\penalty0 (6):\penalty0 1912--1919,
  2010.
\newblock \doi{10.1016/j.cam.2009.08.041}.

\bibitem[{Bonnet-Ben Dhia} et~al.(2012){Bonnet-Ben Dhia}, Chesnel, and
  Jr.]{Dhia:2012}
A.-S. {Bonnet-Ben Dhia}, L.~Chesnel, and P.~Ciarlet Jr.
\newblock {T}-coercivity for scalar interface problems between dielectric and
  metamaterials.
\newblock \emph{ESAIM: M2AN}, 46:\penalty0 1363--1387, 2012.
\newblock \doi{10.1051/m2an/2012006}.

\bibitem[{Bonnet-Ben Dhia} et~al.(2014{\natexlab{a}}){Bonnet-Ben Dhia},
  Chesnel, and Jr.]{Chesnel:2014}
A.-S. {Bonnet-Ben Dhia}, L.~Chesnel, and P.~Ciarlet Jr.
\newblock {\tt {T}}-coercivity for the {M}axwell problem with sign-changing
  coefficients.
\newblock \emph{Comm. Partial Differential Equations}, 39\penalty0
  (6):\penalty0 1007--1031, 2014{\natexlab{a}}.
\newblock \doi{10.1080/03605302.2014.892128}.

\bibitem[{Bonnet-Ben Dhia} et~al.(2014{\natexlab{b}}){Bonnet-Ben Dhia},
  Chesnel, and Jr.]{Dhia:2014}
A.-S. {Bonnet-Ben Dhia}, L.~Chesnel, and P.~Ciarlet Jr.
\newblock Two-dimensional {M}axwell's equations with sign-changing
  coefficients.
\newblock \emph{Appl. Numer. Math.}, 79:\penalty0 29--41, 2014{\natexlab{b}}.
\newblock \doi{10.1016/j.apnum.2013.04.006}.

\bibitem[Chen et~al.(2017)Chen, Qiu, Shi, and Solano]{Chen:2017}
H.~Chen, W.~Qiu, K.~Shi, and M.~Solano.
\newblock A superconvergent hdg method for the maxwell equations.
\newblock \emph{J. Sci. Comput.}, 70\penalty0 (3):\penalty0 1010--1029, 2017.
\newblock \doi{10.1007/s10915-016-0272-z}.

\bibitem[Chesnel and Jr.(2013)]{Chesnel:2013}
L.~Chesnel and P.~Ciarlet Jr.
\newblock {T}-coercivity and continuous {G}alerkin methods: application to
  transmission problems with sign changing coefficients.
\newblock \emph{Numer. Math.}, 124:\penalty0 1--29, 2013.
\newblock \doi{10.1007/s00211-012-0510-8}.

\bibitem[Chung and Jr.(2013)]{Chung:2013}
E.~T. Chung and P.~Ciarlet Jr.
\newblock A staggered discontinuous {G}alerkin method for wave propagation in
  media with dielectrics and meta-materials.
\newblock \emph{J. Comput. Appl. Math.}, 239:\penalty0 189--207, 2013.
\newblock \doi{10.1016/j.cam.2012.09.033}.

\bibitem[Cockburn et~al.(2008)Cockburn, Dong, and
  Guzm\'{a}n]{Cockburn-Dong:2008}
B.~Cockburn, B.~Dong, and J.~Guzm\'{a}n.
\newblock A superconvergent {LDG}-hybridizable {G}alerkin method for
  second-order elliptic problems.
\newblock \emph{Math. Comp.}, 77\penalty0 (264):\penalty0 1887--1916, 2008.
\newblock \doi{10.1090/S0025-5718-08-02123-6}.

\bibitem[Cockburn et~al.(2009)Cockburn, Gopalakrishnan, and
  Lazarov]{Cockburn:2009}
B.~Cockburn, J.~Gopalakrishnan, and R.~Lazarov.
\newblock Unified hybridization of discontinuous {G}alerkin, mixed, and
  continuous {G}alerkin methods for second order elliptic problems.
\newblock \emph{SIAM J. Numer. Anal.}, 47\penalty0 (2):\penalty0 1319--1365,
  2009.
\newblock \doi{10.1137/070706616}.

\bibitem[Cockburn et~al.(2010)Cockburn, Gopalakrishnan, and
  Sayas]{Cockburn:2010}
B.~Cockburn, J.~Gopalakrishnan, and F.~Sayas.
\newblock A projection--based error analysis of {HDG} methods.
\newblock \emph{Math. Comp.}, 79:\penalty0 1351--1367, 2010.
\newblock \doi{10.1090/S0025-5718-10-02334-3}.

\bibitem[Cockburn et~al.(2014)Cockburn, Dubois, Gopalakrishnan, and
  Tan]{Cockburn:2014}
B.~Cockburn, O.~Dubois, J.~Gopalakrishnan, and S.~Tan.
\newblock Multigrid for an {HDG} method.
\newblock \emph{IMA J. Numer. Anal.}, 34\penalty0 (4):\penalty0 1386--1425,
  2014.
\newblock \doi{10.1093/imanum/drt024}.

\bibitem[Cui and Zhang(2013)]{Cui:2013}
J.~Cui and W.~Zhang.
\newblock {An analysis of HDG methods for the Helmholtz equation}.
\newblock \emph{IMA Journal of Numerical Analysis}, 34\penalty0 (1):\penalty0
  279--295, 2013.
\newblock \doi{10.1093/imanum/drt005}.

\bibitem[Evans(2010)]{Evans:book}
L.~C. Evans.
\newblock \emph{Partial differential equations}, volume~19 of \emph{Graduate
  Studies in Mathematics}.
\newblock American Mathematical Society, Providence, RI, {S}econd edition,
  2010.
\newblock \doi{10.1090/gsm/019}.

\bibitem[Feng et~al.(2016)Feng, Lu, and Xu]{feng:2016}
X.~Feng, P.~Lu, and X.~Xu.
\newblock A hybridizable discontinuous galerkin method for the time-harmonic
  maxwell equations with high wave number.
\newblock \emph{Computational Methods in Applied Mathematics}, 16\penalty0
  (3):\penalty0 429--445, 2016.
\newblock \doi{10.1515/cmam-2016-0021}.

\bibitem[{Fraeijs de Veubeke}(1975)]{veubeke}
B.~M. {Fraeijs de Veubeke}.
\newblock Stress function approach.
\newblock In \emph{World Congress on the Finite Element Method in Structural
  Mechanics}. Bournemouth, 1975.

\bibitem[Galdi(2011)]{Galdi:2011}
G.~P. Galdi.
\newblock \emph{An introduction to the mathematical theory of the
  {N}avier--{S}tokes equations. {S}teady-state problems}.
\newblock Springer Monographs in Mathematics. Springer, New York, {S}econd
  edition, 2011.
\newblock \doi{10.1007/978-0-387-09620-9}.

\bibitem[Griesmaier et~al.(2011)Griesmaier, , and Monk]{Griesmaier:2011}
R.~Griesmaier, , and P.~Monk.
\newblock Error analysis for a hybridizable discontinuous {G}alerkin method for
  the {H}elmholtz equation.
\newblock \emph{J. Sci. Comput.}, 49\penalty0 (3):\penalty0 291--310, 2011.
\newblock \doi{10.1007/s10915-011-9460-z}.

\bibitem[Sch\"oberl(2014)]{Schoberl:2014}
J.~Sch\"oberl.
\newblock {C}++11 implementation of finite elements in {NGS}olve.
\newblock Technical Report ASC Report 30/2014, Institute for Analysis and
  Scientific Computing, Vienna University of Technology, 2014.
\newblock URL
  \url{http://www.asc.tuwien.ac.at/~schoeberl/wiki/publications/ngs-cpp11.pdf}.

\bibitem[Stenberg(1991)]{Stenberg:1991}
R.~Stenberg.
\newblock Postprocessing schemes for some mixed finite elements.
\newblock \emph{RAIRO Mod\'{e}l. Math. Anal. Num\'{e}r.}, 25\penalty0
  (1):\penalty0 151--167, 1991.
\newblock \doi{10.1051/m2an/1991250101511}.

\end{thebibliography}
\end{document}